\documentclass[10pt,journal,compsoc]{IEEEtran}
\usepackage{amssymb}
\usepackage{amsmath}
\usepackage{epsfig}
\usepackage{epstopdf}
\usepackage{array}
\usepackage{cases}
\usepackage{color}
\usepackage{multirow}
\usepackage{picinpar}
\usepackage{booktabs}
\usepackage{threeparttable}
\usepackage{longtable}
\usepackage{cases}
\usepackage{lscape}
\usepackage{color}
\usepackage{float}
\usepackage{graphicx}
\usepackage[caption=false,font=scriptsize,labelfont=rm,textfont=rm]{subfig}
\usepackage[justification=centering]{caption}
\usepackage{tikz}
\allowdisplaybreaks[4]

\usepackage{pdflscape}

\usepackage[linesnumbered,ruled,vlined]{algorithm2e}
\usepackage[colorlinks,
            linkcolor=black,
            anchorcolor=green,
            citecolor=blue
            ]{hyperref}

\makeatletter
\@addtoreset{footnote}{section}

\graphicspath{ {./images/} }

\ifCLASSOPTIONcompsoc
  \usepackage[nocompress]{cite}
\else
  \usepackage{cite}
\fi
\ifCLASSINFOpdf

\else

\fi

\begin{document}
\title{A Customized Augmented Lagrangian Method for Block-Structured Integer Programming}

\author{Rui Wang,
  Chuwen Zhang,
  Shanwen Pu, Jianjun Gao, Zaiwen Wen
  \IEEEcompsocitemizethanks{\IEEEcompsocthanksitem Rui Wang and Zaiwen Wen are with the Beijing International Center for Mathematical Research, Peking University, Beijing 100871, China (email: ruiwang@bicmr.pku.edu.cn; wenzw@pku.edu.cn). 
   \IEEEcompsocthanksitem Chuwen Zhang, Shanwen Pu and Jianjun Gao are with the School of Information Management and Engineering, Shanghai University of Finance and Economics, Shanghai 200433, China (email: chuwzhang@gmail.com; 2019212802@live.sufe.edu.cn;gao.jianjun@shufe.edu.cn).}
\thanks{This research is supported in part by the National Natural Science Foundation of China (NSFC) grants 12331010, 72150001, 72225009, 72394360, 72394365.}
 }

%

%

\IEEEtitleabstractindextext{%
  \begin{abstract}
    Integer programming with block structures has received considerable attention recently and is widely used in many practical applications such as train timetabling and vehicle routing problems. It is known to be NP-hard due to the presence of integer variables.  We define a novel augmented Lagrangian function by directly penalizing the inequality constraints and establish the strong duality between the primal problem and the augmented Lagrangian dual problem. Then,  a customized augmented Lagrangian method is proposed to  address the block-structures. In particular,  the minimization of the augmented Lagrangian function is decomposed into multiple subproblems by decoupling the linking constraints and these subproblems can be efficiently solved using the block coordinate descent method. We also establish the convergence property of the proposed method. To make the algorithm more practical, we further introduce several refinement techniques to identify high-quality feasible solutions. Numerical experiments on a few interesting scenarios show that our proposed algorithm often achieves a satisfactory solution and is quite effective.
  \end{abstract}

  \begin{IEEEkeywords}
Integer programming, augmented Lagrangian method, block coordinate descent, convergence
  \end{IEEEkeywords}}

\newcommand{\defi}{\stackrel{\Delta}{=}}
\newcommand{\qed}{\hphantom{.}\hfill $\Box$\medbreak}
\newcommand{\proof}{\noindent{\bf Proof \ }}
\newcommand\one{\hbox{1\kern-2.4pt l }}
\newcommand{\Item}{\refstepcounter{Ictr}\item[(\theIctr)]}
\newcommand{\QQ}{\hphantom{MMMMMMM}}

\newtheorem{theorem}{Theorem}[section]
\newtheorem{condition}{Condition}[section]
\newtheorem{lemma}{Lemma}[section]
\newtheorem{pro}{Proposition}[section]
\newtheorem{teorem}{Theorem}[section]
\newtheorem{corollary}{Corollary}[section]
\newtheorem{definition}{Definition}[section]
\newtheorem{remark}{Remark}[section]
\newtheorem{assumption}{Assumption}[section]
\newtheorem{example}{Example}[section]
\newenvironment{cproof}

{\renewcommand{\qed}{\hfill $\Diamond$} \end{proof}}
\newcommand{\erhao}{\fontsize{21pt}{\baselineskip}\selectfont}
\newcommand{\xiaoerhao}{\fontsize{18pt}{\baselineskip}\selectfont}
\newcommand{\sanhao}{\fontsize{15.75pt}{\baselineskip}\selectfont}
\newcommand{\sihao}{\fontsize{14pt}{\baselineskip}\selectfont}
\newcommand{\xiaosihao}{\fontsize{12pt}{\baselineskip}\selectfont}
\newcommand{\wuhao}{\fontsize{10.5pt}{\baselineskip}\selectfont}
\newcommand{\xiaowuhao}{\fontsize{9pt}{\baselineskip}\selectfont}
\newcommand{\liuhao}{\fontsize{7.875pt}{\baselineskip}\selectfont}
\newcommand{\qihao}{\fontsize{5.25pt}{\baselineskip}\selectfont}

\makeatletter
\newcommand{\figcaption}{\def\@captype{figure}\caption}
\newcommand{\tabcaption}{\def\@captype{table}\caption}
\makeatother

\newcounter{Ictr}

\renewcommand{\theequation}{
  \arabic{equation}}
\renewcommand{\thefootnote}{\fnsymbol{footnote}}

\def\A{\boldsymbol{A}}
\def\B{\boldsymbol{B}}
\def\C{\mathcal{C}}
\def\V{\mathcal{V}}
\def\I{\mathcal{I}}
\def\Y{\mathcal{Y}}
\def\P{\mathbb{P}}
\def\X{\mathcal{X}}
\def\J{\mathcal{J}}
\def\Q{\mathcal{Q}}
\def\W{\mathcal{W}}
\def\S{\mathcal{S}}
\def\T{\mathcal{T}}
\def\M{\mathcal{M}}
\def\N{\mathbb{N}}
\def\R{\mathbb{R}}
\def\H{\mathbb{H}}
\def\s.t.{\text{s.t.}}
\def\a{\boldsymbol{a}}
\def\x{\textbf{x}}
\def\z{\textbf{z}}
\def\u{\textbf{u}}
\def\v{\textbf{v}}
\def\1{\textbf{1}}
\def\bx{\textbf{\rm \x}}
\def\c{\boldsymbol{c}}
\def\b{\boldsymbol{b}}
\def\d{\boldsymbol{d}}
\def\UB{\text{UB}}
\def\LB{\text{LB}}
\def\OV{\text{OV}}
\def\l{\mu}

\maketitle

\IEEEdisplaynontitleabstractindextext

\IEEEpeerreviewmaketitle

\IEEEraisesectionheading{\section{Introduction}\label{sec:introduction}}

\IEEEPARstart{I}{n} this paper, we consider a block-structured integer programming problem:
\begin{subequations}\label{bLP}
  \begin{align}
    \min        & ~~ \c^{\top}\x\label{bLP1}                           \\
    \text{s.t.} & ~~  \A \x \leq \b, \label{bLP2}                      \\
                & ~~ \x_j \in \mathcal{X}_j, ~j=1,2,...,p,\label{bLP3}
  \end{align}
\end{subequations}
where $\x_j \in \R^{n_j}$ is the $j$-th block variable of $\x \in \R^{n}$, i.e., $\x = (\x_1;...;\x_p)$ for $p \geq 1$ with $n=\sum_{j=1}^p n_j$. In \eqref{bLP}, $\A \in \R^{m\times n}, \b \in  \R^{m}, \c \in  \R^{n}$ and the constraint $\X_j$ is the set of $0,1$ vectors in a polyhedron, i.e.,
$$\X_j:=\{\x_j  \in \{0,1\}^{n_j}: \B_j\x_j \leq \d_j\},~j=1,2,...,p.$$
The constraints \eqref{bLP3} can be reformulated as $\x \in \mathcal{X}:=\{\x \in \{0,1\}^n:~\B\x \leq \d\}$ where the block diagonal matrix $\B \in \R^{q\times n}$ is formed by the small submatrices $\B_j$ as the main diagonal and $\d=(\d_1;...;\d_p)$. Correspondingly, $\c$ and $\A$ can be rewritten as $\c=(\c_1;\c_2;...;\c_p)$ with $\c_j \in \R^{n_j}$ and $\A=(\A_1~ \A_2~...~\A_p)$ with $\A_j \in \R^{m \times n_j}$.

Assume that these constraints \eqref{bLP3} are ``nice" in the sense that an integer program with just these constraints is easy. Therefore, if the coupling constraints \eqref{bLP2} are ignored, the remaining problem which is only composed of the constraints \eqref{bLP3} is easier to solve than the original problem \eqref{bLP}.
For convenience, we assume $\X$ is not empty and \eqref{bLP} is feasible. Denote  by $f^{\text{IP}}$ the optimal value of the problem \eqref{bLP}. 
 This block structure is closely related to an important model known as ``$n$-fold integer programming (IP)" studied extensively in computer vision \cite{wang2016large,lam2020fast}, machine learning \cite{lin2019approximate,anderson2020strong} and theoretical computer science \cite{eisenbrand2019algorithmic,cslovjecsek2024parameterized}, etc. The theoretical foundations of $n$-fold IPs have significant implications for efficient algorithm development in various fields. For example, an algorithmic theory of integer programming based on $n$-fold IP was proposed in \cite{eisenbrand2019algorithmic}. Recent advancements have been provided in \cite{cslovjecsek2024parameterized}. Furthermore, the progress in theory and application of integer programming with a block structure was summarized in \cite{de2008n, Chen2019}. While existing works on $n$-fold IP mainly focus on asymptotic analyses, this work aims to develop an efficient augmented Lagrangian approach tailored to the block structure for practical efficiency with a convergence guarantee.

\subsection{Related Work}

The branch and bound algorithm for general integer programming (IP) was first introduced by Land and Doig \cite{land2010automatic}. Gomory \cite{gomory1958algorithm} developed a cutting plane algorithm for integer programming problems. These two approaches are at the heart of current state-of-the-art software for integer programming.
Unfortunately, these methods often suffer from high computational burdens due to the discrete constraint, thus they are not good choices for solving some large-scale practical problems. Therefore, it is necessary to develop efficient approaches to obtain feasible and desirable solutions, even if they may not be globally optimal. Considering the block structure of the integer linear programming \eqref{bLP}, a natural idea is to decompose a large global problem into smaller local subproblems. There are three typical decomposition methods for solving such problem: Benders \cite{bnnobrs1962partitioning}, Dantzig-Wolfe (DW) \cite{dantzig1960decomposition} and Lagrangian decompositions \cite{geoffrion1974lagrangean}. The Benders decomposition method is used to deal with mixed integer programming problems by decomposing them into a master problem and a primal subproblem. It generates cuts that are added to the master problem by solving the dual of the primal subproblem. This method is not suitable for \eqref{bLP}, since our primal subproblem is an integer programming problem, which is still NP-hard. The DW decomposition method can solve block structured integer programming problem \eqref{bLP} based on resolution theorem. To improve the tractability of large-scale problems, the DW decomposition relies on column generation. However, it might be harder or even intractable to solve the master problem by column generation if further constraints are applied to $\X_j$ \cite[Chapter 8.2]{conforti2014integer}.  The Lagrangian decomposition introduces Lagrange multipliers and constructs a sequence of simpler subproblems. 
 Although the method itself has a limitation due to its inherent solution symmetry for certain practical problems,
 the Lagrangian duality bound is  useful in  the branch-and-bound procedure. A Lagrangian heuristic algorithm has also been proposed in \cite{CAPRARA2006738} for solving a real-world train timetabling problem, which decomposed the problem into smaller subproblems using Lagrangian relaxation and developed a heuristic algorithm based on subgradient optimization to generate feasible solutions.

 Many techniques from continuous optimization including the alternating direction method of multipliers (ADMM)  have been applied to  solve integer programming recently.
The authors in \cite{wu2018ell} proposed an $\ell_p$-box ADMM for solving a binary integer programming, where the binary constraint is replaced by the intersection of a box and an $\ell_p$-norm sphere. The authors in \cite{sun2021decomposition} proposed augmented Lagrangian method (ALM) and ADMM based on the $\ell_1$ augmented Lagrangian function for two-block mixed integer linear programming (MILP). For the multi-block MILP problem,  an extended ADMM was proposed in \cite{feizollahi2015large} to employ a release-and-fix approach to solve the subproblems. There are also some heuristic methods based on ADMM \cite{takapoui2017alternating,yao2019admm,ZHANG2020102823,kanno2018alternating} that have been successfully applied to various practical integer programming problems.


There are some other widespread approaches on the relaxation of the binary constraint including linear programming (LP) relaxation \cite{johnson2000progress,feldman2005using} and semidefinite relaxation \cite{laurent2005semidefinite,wang2013fast}.
For the multi-block MILP problem, several inexact methods have been proposed including a distributed algorithm relying on primal decomposition \cite{camisa2018primal}, a dual decomposition method \cite{vujanic2016decomposition}, and a decomposition-based outer approximation method \cite{muts2020decomposition}.
By introducing continuous variables to replace the discrete variables,  the exact penalty methods \cite{murray2010algorithm,lucidi2010exact,yu2013exact} have been studied for solving the nonlinear IP problems.  Then the problem is transformed into an equivalent nonlinear continuous optimization problem.

\subsection{Contributions}
In this paper, we propose a customized ALM for solving \eqref{bLP}. Our main contributions are listed below.

\begin{itemize}
  \item[(i)] We define a novel augmented Lagrangian (AL) function that differs from the classical AL function and establish strong duality theory for the augmented Lagrangian relaxation of the block-structured integer programming \eqref{bLP}, which motivates us to utilize the ALM for solving the problem \eqref{bLP}.
  \item[(ii)] Based on the special structure of \eqref{bLP}, we propose two block coordinate descent (BCD)-type methods that are well-suited for solving the resulting subproblems in our ALM framework. These methods utilize classical update and proximal linear update techniques, denoted as ALM-C and ALM-P, respectively. We also analyze their convergence properties under proper conditions.
  \item[(iii)] To address challenges in finding the global optimal solution for practical problems such as train timetabling, we introduce refinement strategies and propose a customized ALM to enhance the quality of solutions generated by the ALM. Our numerical experiments demonstrate the effectiveness of the proposed method in solving large-scale instances. 
\end{itemize}

 Note that the ADMM-based method in \cite{ZHANG2020102823} solves each subproblem only once per iteration for a fixed Lagrangian multiplier. On the other hand, our ALM method  involves multiple iterations using the BCD method to minimize each AL function until certain rules are satisfied. Once the AL subproblem is solved exactly, the strong duality guarantees that the ALM can converge to a global minimizer of the problem \eqref{bLP}. Therefore, achieving high accuracy in minimizing the AL function allows our ALM to achieve superior solutions with fewer iterations, resulting in significantly reduced computational time compared to the ADMM. This claim is supported by our numerical tests. Additionally, we introduce Assumptions 3.1 and 3.2, derived from the structural characteristics of the practical problem in \cite{yao2019admm} and \cite{ZHANG2020102823}, and subsequently analyze the theoretical properties of the ALM-C method under these assumptions. However, our ALM-P method has wider applicability and does not require these specific assumptions. Moreover, the ALM can be regarded as a dual ascent algorithm with respect to dual variables, hence ensuring its convergence. In contrast,  the convergence analysis of the ADMM for solving this kind of problem remains unclear.
 
The existing ALM-based methods for solving integer programming in \cite{feizollahi2017exact} mainly focused on the duality of the augmented Lagrangian dual for MILP with equality constraints, but  numerical experiments were not available. Moreover, our approach differs significantly in that we handle inequality constraints directly, without introducing slack variables. This approach has two key benefits: it reduces the number of variables, thus decreasing computational burden in high-dimensional cases, and it allows for more customized algorithmic design based on the inherent structure of the problem.
 In \cite{sun2021decomposition}, the $\ell_1$ norm was considered as an augmented term in the AL function for MILP such that the minimization of the AL function can be decomposed into multiple low-dimensional $\ell_1$-penalty subproblems due to the separable blocks. These subproblems were then solved in parallel using the Gurobi solver. In our work, we take a different approach by using a quadratic term in the AL function which allows the augmented Lagrangian subproblem to be reduced to a linear programming under certain conditions. Furthermore, we update the blocks of variables sequentially one at a time.

\subsection{Notation and Organization}
Let $\N_m:=\{1,2,...,m\}$, $\N_p:=\{1,2,...,p\}$ and $\R^m_+:= \{x \in \R^m: x_i \geq 0~\text{for all}~ i\}$.  The superscript $``\top"$ means ``transpose". Denote $\a_i$ by the $i$-th row of the matrix $\A$ and $b_i$ by the $i$-th element of the vector $\b$.
For convenience, we let $\A_{\I,j}$ denote the submatrix consisting of columns of $\A_j$ indexed by $\I$ and $\b_{\I}$ denote the subvector consisting of entries of $\b \in \R^m$ indexed by $\I$.
A neighborhood of a point $\x^*$ is a set $\mathcal{N}(\x^*,1)$ consisting of all points $\x$ such that $\|\x-\x^*\|^2 \leq 1$. Let $\mathbf{1}$ be a row vector of all ones.

This paper is structured as follows.  An AL function is defined and the strong duality of the problem \eqref{bLP} is discussed in section \ref{sec_Duality}. We propose a customized ALM incorporating BCD methods and refinement strategies to improve the quality of the ALM solutions in section \ref{sec_ALM}.  We establish the convergence results of both the BCD methods to minimize the AL function and the ALM applied to the whole problem \eqref{bLP} in section \ref{sec_con}.  The proposed method is applied to two practical problems in section \ref{sec_app}. Concluding remarks are made in the last section.
\section{The AL Strong Duality}\label{sec_Duality}
The Lagrangian relaxation (LR) of (\ref{bLP}) with respect to the constraint \eqref{bLP2} has the following form:
\begin{equation}\label{LR}
  \min_{\x \in \X}~~ \mathcal{L}(\bx,\lambda)  :=\sum\limits_{j=1}^p~ \c_j^{\top}\x_j+\lambda^{\top}\left(\sum\limits_{j=1}^p\A_j\x_j - \b\right),
\end{equation}
where  $\lambda \in \R^m_+$ is a Lagrange multiplier associated with the constraint \eqref{bLP2}. We can observe that \eqref{LR} is much easier to be solved than the original problem \eqref{bLP} since  \eqref{LR} can be decomposed into $p$-block low-dimensional independent subproblems. However, there may exist a non-zero duality gap when certain constraints are relaxed by using classical Lagrangian dual \cite{feizollahi2017exact}. To reduce the duality gap, we add a quadratic penalty function to the  Lagrangian function in \eqref{LR} and solve the augmented Lagrangian dual problem which is defined as follows.
\begin{definition}[AL Dual]
  We define an AL function by
  \begin{eqnarray}\label{alf}
    L(\bx,\lambda,\rho)  &=&\sum\limits_{j=1}^p~ \c_j^{\top}\bx_j+\lambda^{\top}\left(\sum\limits_{j=1}^p\A_j\bx_j - \b\right)\nonumber\\
    && +\frac{\rho}{2}\left\|\left(\sum\limits_{j=1}^p\A_j\bx_j - \b\right)_+\right\|^2,
  \end{eqnarray}
  where $\lambda \in \R_+^m, \rho>0.$
  The  corresponding AL relaxation of \eqref{bLP} is given by
  \begin{equation}\label{df}
    d(\lambda, \rho) := \min_{\bx \in \X} ~L(\bx,\lambda,\rho).
  \end{equation}
  We call the following maximization problem the AL dual problem:
  \begin{equation}\label{zdp}
    f_\rho^{\text{LD}}:=\max_{\lambda\in \R_{+}^m} d(\lambda, \rho).
  \end{equation}
\end{definition}
Note that the classical AL function of (\ref{bLP}) is given by
\begin{equation}\label{cal}
  \hat{L}(\bx,\lambda,\rho)= \c^{\top}\bx+\frac{\rho}{2}\left\|\left(\A\bx - \b+\frac{\lambda}{\rho}\right)_+\right\|^2-\frac{\|\lambda\|^2}{2\rho}.
\end{equation}
We prefer using the form of the AL function \eqref{alf} rather than the classical version \eqref{cal} due to the absence of $\lambda/\rho$ in the quadratic term of the max function. This makes it possible to convert the AL function into a linear function under certain conditions, making the problem \eqref{df} easier to solve, which will be explained in the next section. We also verify that the quadratic term in \eqref{alf} is an exact penalty and strong duality holds between the AL dual problem \eqref{zdp} and the primal problem \eqref{bLP}.

\begin{lemma}[Strong Duality]\label{strongd}
  Suppose the problem \eqref{bLP} is feasible and its optimal value is bounded.   If a minimum achievable non-zero slack exists, i.e., there is a $\delta$ such that for any $ i \in\N_m$,
  \begin{equation}\label{delta}
    0<\delta \leq \min_{\bx \in \X}\{(\a_i\bx-b_i)^2:  \a_i\bx> b_i\},
  \end{equation}
  then there exists a finite $\rho^* \in  (0, +\infty)$ such that
  $$f_{\rho^*}^{\text{LD}} = \min_{\A{\textbf{\rm \x}} \leq \b, {\textbf{\rm \x}} \in \mathcal{X}} \c^{\top}{\textbf{\rm \x}}. $$
\end{lemma}
\begin{proof}
 For any $\lambda \in  \R_+^m $ and $\rho > 0$, since $\{\x \in \X : \A\x - \b \leq 0\} \subseteq \X$, we have
\begin{equation}\label{fip}
  d(\lambda, \rho) \leq  \mathop{\min}_{\x \in \X \atop \A\x - \b\leq 0} ~L(\x,\lambda,\rho)
  \leq \mathop{\min}_{\x \in \X \atop \A\x - \b \leq 0} \c^{\top}\x =f^{\text{IP}}.
\end{equation}
Then $f_\rho^{\text{LD}} \leq f^{\text{IP}}$.

Now it suffices to find a finite $\rho^*$ such that $f_{\rho^*}^{\text{LD}} \geq f^{\text{IP}}$. We first let $\x^0$ be any arbitrary feasible solution of \eqref{bLP}, that is, $\x^0 \in \X$ and
$\A\x^0 \leq \b$. Denote by $f^{\text{LP}}$ the linear programming (LP) relaxation of $f^{\text{IP}}$.
Since the value of the LP relaxation of (\ref{bLP}) is bounded \cite{blair1977value}, i.e., $-\infty < f^{\text{LP}}\leq \c^{\top}\x^0 < +\infty$, we set $\rho^*=2( \c^{\top}\x^0-f^{\text{LP}})/\delta$, then  $0 < \rho^* < +\infty$. Moreover,
\begin{equation}\label{fld}
  \!f^{\text{LD}} \geq \max_{\lambda\in \R_+^m} d(\lambda, \rho^*)\geq d(\lambda^*, \rho^*) =\min_{\x \in\X} L(\x, \lambda^*, \rho^*),
\end{equation}
where $\lambda^* \in \R_+^m$ is a given parameter.
Let $\mathbb{I}:=\{i \in \N_m: \a_i\x-b_i > 0, \x \in \X\}$, we consider following two cases:

Case 1: $\mathbb{I}= \emptyset$. In this case we have $\A\x \leq \b$ for all $\x \in \X$.  By letting $\bar{\lambda}=0$, we can obtain that
\begin{align}\label{fip1}
  L(\x, \bar{\lambda}, \rho^*) & =\c^{\top}\x+\bar{\lambda}^{\top}(\A\x - \b)+\frac{\rho^*}{2}\|\left(\A\x - \b\right)_+\|^2 \nonumber \\
                                     & = \c^{\top}\x                                   
                                    \geq  f^{\text{IP}}.
\end{align}

Case 2: $\mathbb{I}\neq \emptyset$. Denote by $\lambda^{\text{LP}}$  a positive optimal vector of dual variables for $\A\x \leq \b$ in the LP relaxation of (\ref{bLP}). In this case we get
\begin{align*}
  L(\x, \lambda^{\text{LP}}, \rho^*)
    =&\c^{\top}\x+(\lambda^{\text{LP}})^{\top}(\A\x - \b)+\frac{\rho^*}{2}\sum_{i \in \mathbb{I}}(\a_{i}\x - b_{i})^2\\
    &+\frac{\rho^*}{2}\sum\nolimits_{i \notin \mathbb{I}}((\a_{i}\x - b_{i})_+)^2\nonumber \\
    =&\c^{\top}\x+(\lambda^{\text{LP}})^{\top}(\A\x - \b)+\frac{\rho^*}{2}\sum_{i \in \mathbb{I}}(\a_{i}\x - b_{i})^2.
\end{align*}
Since
$$\frac{\rho^*}{2}\sum_{i \in \mathbb{I}}(\a_{i}\x - b_{i})^2 \geq \frac{\rho^*}{2}\min_{i \in \mathbb{I}}(\a_{i}\x - b_{i})^2 \overset{\eqref{delta}}{\geq} \frac{\rho^*}{2} \delta = \c^{\top}\x^0-f^{\text{LP}},$$
it yields
\begin{align}\label{fip2}
  L(\x, \lambda^{\text{LP}}, \rho^*)
   & \geq \c^{\top}\x+(\lambda^{\text{LP}})^{\top}(\A\x - \b)+( \c^{\top}\x^0-f^{\text{LP}}) \nonumber \\
   & \geq f^{\text{LP}}+( \c^{\top}\x^0-f^{\text{LP}})= \c^{\top}\x^0 \geq f^{\text{IP}},
\end{align}
where the second inequality holds due to the definition of $\lambda^{\text{LP}}$.
Thus the inequalities \eqref{fip1} and \eqref{fip2} by letting $\lambda^*$ be $\lambda^{\text{LR}}$ and $\bar{\lambda}$ imply that
$$d(\lambda^*, \rho^*) = \min\nolimits_{\x \in \X} ~L(\x,\lambda^*, \rho^*) \geq f^{\text{IP}}. $$
This together with \eqref{fld} and \eqref{fip} yields that
$$f_{\rho^*}^{\text{LD}}  = d(\lambda^*, \rho^*) = f^{\text{IP}}.$$
Hence we complete the proof.
\qed
\end{proof}

 One can observe from the proof that there exists a finite value $\rho^*$ such that for all $\rho \geq \rho^*$, the strong duality still holds. Therefore,  given a sufficiently large penalty parameter $\rho$, we can achieve a satisfactory feasibility of \eqref{bLP}.  Specifically, as $\rho$ increases beyond $\rho^*$, the augmented Lagrangian method penalizes constraint violations \eqref{bLP2} more heavily, thereby making the solution closer to the feasible region of the problem. The strong duality allows us to obtain a globally optimal solution to the problem \eqref{bLP}  by solving the augmented Lagrangian dual problem \eqref{zdp}.

To utilize the concave structure of the dual function $d(\lambda,\rho)$, we apply the projected subgradient method for solving \eqref{zdp} since the dual function  is not differentiable. We first give the definition of subgradient and subdifferential.
\begin{definition}\label{def_subg}
  Let $h: \R^m \rightarrow \R$ be a convex function. The vector $s \in \R^m$ is called a subgradient of $h$ at $\bar{x} \in \R^m$ if
  $$h(x)-h(\bar{x}) \geq s^{\top}(x-\bar{x}),~~\forall x\in\R^m.$$
  The subdifferential of $h$ at $\bar{x}$ is the set of all subgradients of $h$ at $\bar{x}$ which is given by
  $$\partial h(x)=\{s\in \R^m: h(x)-h(\bar{x}) \geq s^{\top}(x-\bar{x}),  ~~\forall x\in\R^m\}.$$
\end{definition}
Since $d(\lambda,\rho)$ is concave, we adjust Definition \ref{def_subg} to correspond to the set $-\partial (-d(\lambda,\rho))$, allowing us to apply properties of the subdifferential of a convex function to a concave function.
\begin{pro}\label{subgra}
  Consider the dual function $d(\lambda, \rho): \R^{m+1} \rightarrow \R$. Then the subdifferentials of $d$ at $\lambda$ and $\rho$ satisfy
  \begin{equation}\label{subg}
    \A\bx -\b \in \partial_{\lambda} d(\lambda,\rho)
    ,\quad
    \frac{1}{2}\|(\A\bx -\b)_+\|^2 \in \partial_{\rho} d(\lambda,\rho),
  \end{equation}
  where $\bx$ is a solution of \eqref{df} with input $(\lambda, \rho)$.
\end{pro}
\begin{proof}
Let $\x$ be a solution of \eqref{df} with input $(\lambda, \rho)$. Then for any pair $(\hat{\lambda}, \hat{\rho}) \in \R^m_+ \times \R_+$, it holds that
\begin{eqnarray*}
  \begin{array}{ll}
    d(\hat{\lambda}, \hat{\rho})\!\!\!\! & \leq \c^{\top}\x +\hat{\lambda}^{\top}(\A\x  -\b) + \frac{\hat{\rho}}{2}\|(\A\x  -\b)_+\|^2                                                                                                      \\
                                         & = \c^{\top}\x +\lambda^{\top}(\A\x -\b) + \frac{\rho}{2}\|(\A\x -\b)_+\|^2                                 \\&\quad+(\hat{\lambda}-\lambda)^{\top}(\A\x -\b)+\frac{\hat{\rho}-\rho}{2}\|(\A\x -\b)_+\|^2 \\
                                         & = d(\lambda, \rho)+(\hat{\lambda}-\lambda)^{\top}(\A\x -\b) + \frac{\hat{\rho}-\rho}{2}\|(\A\x -\b)_+\|^2,
  \end{array}
\end{eqnarray*}
where the last equality holds due to the optimality of $\x$. Then by the definition of subgradient and subdifferential, we arrive at \eqref{subg}.
\qed
\end{proof}

\section{Augmented Lagrangian method}\label{sec_ALM}

In this section, we introduce an augmented Lagrangian method framework. This method is composed of two steps in solving the augmented Lagrangian dual problem \eqref{zdp}.  We first use the primal information $\x$ to construct the subgradient of $d$ at $(\lambda, \rho)$. Since $\lambda$ and $\rho$ satisfy  $\lambda \geq 0$ and $\rho>0$, then we apply the projected subgradient method to update parameters $\lambda$ and  $\rho$.  Starting from $\lambda^0 \in \R_+^{m}$ and $\rho^0>0$, we can update $\lambda$ and $\rho$ at $(k+1)$-th iteration by
\begin{subequations}
  \begin{align}
    \lambda^{ k + 1} & = \left(\lambda^k+\alpha^k (\A\x^{k+1}-\b)\right)_+,\label{lam}                                                                                          \\
    \rho^{k+1}       & = \left(\rho^k+\frac{\alpha^k}{2}\left\|(\A\x^{k+1} -\b)_+\right\|^2\right)_+ \nonumber\\
    &= \rho^k+\frac{\alpha^k}{2}\left\|(\A\x^{k+1} -\b)_+\right\|^2,\label{rho}
  \end{align}
\end{subequations}
where $\x^{k+1}$ is an optimal solution of the augmented Lagrangian relaxation problem \eqref{df} at $\lambda^k$ and $\rho^k$, that is,
\begin{equation}\label{xsub}
  \x^{k+1} \in \arg\min\limits_{\x \in  \mathcal{X}}~ L (\x,  \lambda^k,\rho^k).
\end{equation}
Therefore, the iterative processes \eqref{xsub}, \eqref{lam} and \eqref{rho} consisting of primal and dual variables make up the ALM framework.

Solving the $\x$-subproblem \eqref{xsub} is an important step in ALM. While the subproblem \eqref{xsub} has no closed form solution in general, we can apply an iterative method to solve it exactly or inexactly. Consequently, the ALM for solving the problem \eqref{bLP} consists of outer and inner iterations. The subscript $k$ is used to denote
the outer iteration number, and the subscript $t$ is used to denote the inner iteration number.

\subsection{A BCD method for subproblem \eqref{xsub}}

In this subsection, we present a BCD method for solving the $\x$-subproblem \eqref{xsub} in the ALM framework. The BCD method minimizes the function $L$ by iterating cyclically in order $\x_1,...,\x_p$, fixing the previous iteration during each iteration. Denote $\x^t = \left(\x_1^t; \x_2^t;...;\x_p^t\right)$ where $\x^t_j$ is the value of $\x_j$ at its $t$-th update. Let
$$ L^t_j(\x_j,\lambda,\rho)= L(\x^{t+1}_1,...,\x^{t+1}_{j-1},\x_j,\x^t_{j+1},...,\x^t_{p},\lambda,\rho),$$
and
$$\x^t(j) = \left(\x_1^{t+1},\x_2^{t+1},...,\x_{j-1}^{t+1},\x_{j}^{t},\x_{j+1}^{t},...,\x_{p}^{t}\right).$$
Therefore $\x^t(1) = \left(\x_1^t; \x_2^t;...;\x_p^t\right)=\x^t$ and $\x^t(p+1) = \left(\x_1^{t+1}; \x_2^{t+1};...;\x_p^{t+1}\right)=\x^{t+1}.$ Given fixed parameters $\lambda\geq 0$ and $\rho>0$,
we can also calculate the gradient of $L(\x^t(j),\lambda,\rho)$ at $\x_j$ by
  \begin{eqnarray*}
  g_j(\x^t)&:=&\nabla_{\x_j} L(\x^t(j),\lambda,\rho)\\
  &=&\c_j+\A_j^{\top}\lambda+\rho\A_j^{\top}\left(\A\x^t(j) - \b\right)_+.
\end{eqnarray*}
At each step,  we consider two types of updates for every $\x_j \in \X_j$:
\begin{subequations}
\begin{align}
  & \text{Classical:}~  \x^{t+1}_j \in \mathop{\arg\min}_{\x_j \in \X_j}~ L^t_j(\x_j,\lambda,\rho),  \label{xjko}\\
  &\text{Proximal linear:}\nonumber \\
  & \x_j^{ t + 1}  \in \mathop{\arg\min}_{\x_j \in  \mathcal{X}_j} \left\{\langle \x_j-\x_j^t, g_j(\x^t)\rangle + \frac{1}{2\tau}\|\x_j-\x_j^t\|^2\right\},\label{xjkl}
  \end{align}
\end{subequations}
where $\tau>0$ is a step size.
In fact, if one defines the projection operator by
$$P_\X(v) \in \arg\min\{\|u-v\|: u \in \X\},$$
then we obtain the following equivalent form of \eqref{xjkl}:
$$ \x_j^{ t + 1}=P_{\X_j}\left(\x_j^t-\tau g_j(\x^t)\right).$$

In general, the classical subproblem \eqref{xjko} is fundamentally harder to solve because of the quadratic term in the objective function. However,  we derive a simplified form of this subproblem under certain conditions, which will be discussed in Subsection \ref{BCD_cu}. By contrast, the prox-linear subproblem \eqref{xjkl} is relatively easy to solve because the objective function is linear with respect to $\x_j$. Now we summarize the ALM for solving  \eqref{bLP} in Algorithm \ref{ALM}, which allows each $\x_j$ to be updated by \eqref{xjko} or \eqref{xjkl}. 
We assume that each block $j$ is updated by the same scheme in \eqref{xjko} and \eqref{xjkl} for all iteration $t$.

\begin{algorithm}
  \caption{ ALM with BCD}
  \label{ALM}
  \KwIn{Initial point $\x^0$, $\lambda^0,\rho^0$.}
  \KwOut{A feasible solution $\x^{k+1}$.}
  \For{ $k=0,1,...,k_{\max}$}
  {
  \For{$t=0,1,...,t_{\max}$}
  {\For{$j=1,2,...,p$}
  {Compute $\x_j^{(t+1)}$ by \eqref{xjko} or \eqref{xjkl}\;}
  \textbf{if} $\bx^{(t+1)}=\bx^{(t)}$, \textbf{then} let $\x^{k+1}=\x^{(t+1)}$ and \textbf{break}\;}
  \textbf{if} $\|(\A\bx^{k+1}-\b)_+\|_2 =0$, \textbf{then Teminate}\;
  Update the Lagrangian multipliers  $\lambda^{k+1}$  by \eqref{lam} and the penalty coefficient $\rho^{k+1}$ by \eqref{rho}.
  }
\end{algorithm}

Subsequently, we show more detailed information about these two updates \eqref{xjko} and \eqref{xjkl}. 

\subsubsection{Proximal linear update of BCD}
Before considering the proximal linear subproblem \eqref{xjkl}, we first give a definition of a linear operator, which is essential in the BCD method.
\begin{definition}\label{lo}
  The linear operator associated to a vector $v \in \R^n$ is defined by
  $$\mathcal{T}_{\Omega}(v)=\mathop{\arg\min}_{\u \in  \Omega}~ v^{\top}\u,$$
  where $\Omega$ is a nonempty and closed set.
\end{definition}

Benefiting from the good characteristics of $\x$ being a binary variable, we have
\begin{equation}\label{x1x}
  \|\x_j\|^2=\textbf{1}^{\top}\x_j,~~\text{for}~~\forall j \in \N_p,
\end{equation}
then the objective function in \eqref{xjkl} is linear. Therefore, we can also rewrite \eqref{xjkl} as
\begin{align}\label{clf}
  \x_j^{ t + 1}
  &\in \mathop{\arg\min}_{\x_j \in  \mathcal{X}_j}~ \left\{ \x_j^{\top} g_j(\x^t) + \frac{1}{2\tau}\1^{\top}\x_j-\frac{1}{\tau}\x_j^{\top}\x_j^t\right\}  \nonumber\\
  &=\mathcal{T}_{\X_j}\left(\tau g_j(\x^t)+\frac{\textbf{1}}{2}-\x_j^t\right).
\end{align}
Due to the discrete property of the feasible set $\X_j$, it is crucial to choose the step size $\tau$ appropriately. If $\tau$ is too large, the BCD method will not converge. If $\tau$ is too small, the BCD method will be stuck at some points. The reason why this happens will be explained in the convergence analysis.

\subsubsection{Classical update of BCD}\label{BCD_cu}
We start by giving the following assumption of model \eqref{bLP}, which is very common in many applications such as train timetabling, vehicle routing, and allocation problems.
\begin{assumption}\label{Ab1}
  The entries of the matrix $\A$ are either $1$ or $0$.  The vector $\b$ equals $\textbf{\rm \1}$.
\end{assumption}
Under this assumption, we consider the subproblem \eqref{xjko}. For each $j \in \N_p$,  if the condition $\A_{j}\bx^{t+1}_j \leq \textbf{\rm \1}$  always holds for each iteration of the BCD method, then
  \begin{align}\label{xjkot}
    \bx^{t+1}_j  &\in \mathop{\arg\min}_{\bx_j \in \X_j}~ L^t_j(\bx_j,\lambda,\rho)\\
    &=\mathcal{T}_{\X_j}\left(\c_j+\A_j^{\top}\lambda+\rho\A_j^{\top}\left(\sum\limits_{l\neq j}^p\A_l\bx^t_l(j) -\frac{\textbf{\rm \1}}{2}\right)_{+}\right).\nonumber
  \end{align}
To prove the above derivation, we introduce following two notations at the $t$-th update:
\begin{eqnarray}\label{In}
\begin{array}{ll}
  \I        :=\{i \in \N_m: \sum_{l\neq j}^p\A_{i,l}\x_l^t(j) = 0 \},   \\
  \bar{\I}  :=\{i \in \N_m: \sum_{l\neq j}^p\A_{i,l}\x_l^t(j) \geq 1 \},
  \end{array}
\end{eqnarray}
where $\x^t_l(j)=\left\{\begin{array}{ll}
    \x^{t+1}_l, & \text{if}~ l <j, \\
    \x^t_l,     & \text{if}~ l >j.
  \end{array}\right.$
Obviously, $\I \cap \bar{\I}=\emptyset$ and $\I \cup \bar{\I}=\N_m$. Therefore, we have
\begin{align}
   & \sum\nolimits_{l\neq j}^p\A_{\bar{\I},l}\x^t_l(j) -\frac{\1_{\bar{\I}}}{2}=\left(\sum\nolimits_{l\neq j}^p\A_{\bar{\I},l}\x^t_l(j) -\frac{\1_{\bar{\I}}}{2}\right)_+,\label{AbIj2} \\
   & \sum\nolimits_{l\neq j}^p\A_{\I,l}\x^t_l(j) -\frac{\1_{\I}}{2}=0, \label{AIj2}
\end{align}
and
\begin{eqnarray}\label{AIj}
  \hspace{-3mm}\begin{array}{ll}
    \left(\A_{\I,j}\x_j+\sum\limits_{l\neq j}^p\A_{\I,l}\x^t_l(j) -\1_{\I}\right)_{\!\!+}=\left(\A_{\I,j}\x_j -\1_{\I}\right)_{+}\!=\!\textbf{0},\label{AIj-1} \\
    \A_{\bar{\I},j}\x_j+\sum_{l\neq j}^p\A_{\bar{\I},l}\x^t_l(j) -\1_{\bar{\I}} \geq \textbf{0}.\label{AIj-2}
  \end{array}
\end{eqnarray}
Since the element in $\A_j\x_j$ is either $1$ or $0$, we have
\begin{equation}\label{Ajxj}
  \|\A_j\x_j\|^2=\textbf{1}^{\top}(\A_j\x_j), ~~ \forall~ j\in \N_p.
\end{equation}
Denote $C=\left\|\sum_{l\neq j}^p\A_{\bar{\I},l}\x^t_l(j) -\1_{\bar{\I}}\right\|^2$. Then
\begin{eqnarray}\label{Abar}
\begin{array}{ll}
  &\left\|\left(\A_j\x_j+\sum_{l\neq j}^p\A_{l}\x^t_l(j) -\1\right)_{+}\right\|^2\nonumber\\
  \overset{\eqref{AIj}}{=}&\|\A_{\bar{\I},j}\x_j\|^2+2(\A_{\bar{\I},j}\x_j)^{\top}\left(\sum_{l\neq j}^p\A_{\bar{\I},l}\x^t_l(j) -\1_{\bar{\I}}\right)+C\nonumber\\
  \overset{\eqref{Ajxj}}{=}&2(\A_{\bar{\I},j}\x_j)^{\top}\left(\sum_{l\neq j}^p\A_{\bar{\I},l}\x^t_l(j) -\frac{\1_{\bar{\I}}}{2}\right)+C\nonumber\\
  \overset{\eqref{AbIj2}}{=}&2(\A_{\bar{\I},j}\x_j)^{\top}\left(\sum_{l\neq j}^p\A_{\bar{\I},l}\x^t_l(j) -\frac{\1_{\bar{\I}}}{2}\right)_+ +C\nonumber\\
  \overset{\eqref{AIj2}}{=}& 2(\A_{\bar{\I},j}\x_j)^{\top}\left(\sum_{l\neq j}^p\A_{\bar{\I},l}\x^t_l(j) -\frac{\1_{\bar{\I}}}{2}\right)_+ \\
  &+ 2(\A_{\I,j}\x_j)^{\top}\left(\sum_{l\neq j}^p\A_{\I,l}\x^t_l(j) -\frac{\1_{\I}}{2}\right)_+ +C\nonumber\\
  =&2(\A_j\x_j)^{\top}\left(\sum_{l\neq j}^p\A_l\x^t_l(j) -\frac{\1}{2}\right)_+ +C.
  \end{array}
\end{eqnarray}
Then the iterative scheme for $\x$-update is given by
\begin{eqnarray*}
  \begin{array}{ll}
    \x^{t+1}_j & = \mathop{\arg\min}\limits_{\x_j \in \X_j}~ L^t_j(\x_j,\lambda,\rho)\nonumber                                                                                                                                                                                 \\
               & =\mathop{\arg\min}\limits_{\x_j \in \X_j}~\left\{\c_j^{\top}\x_j+\lambda^{\top}\left(\A_j\x_j - \1\right)        \right.               \\
               & \quad \left.  +\frac{\rho}{2}\left\|\left(\A_j\x_j+\sum_{l\neq j}^p\A_{l}\x^t_l(j) -\1\right)_{+}\right\|^2\right\}\nonumber \\
               & \overset{\eqref{Abar}}{=} \mathop{\arg\min}_{\x_j \in \X_j} \left\{\c_j^{\top}\x_j+\lambda^{\top}\left(\A_j\x_j - \1\right)  \right.               \\
               & \quad \left.
               +\rho(\A_j\x_j)^{\top}\left(\sum_{l\neq j}^p\A_l\x^t_l(j)-\frac{\1}{2}\right)_{+}\right\}\nonumber                 \\
               & =\mathcal{T}_{\X_j}\left(\c_j+\A_j^{\top}\lambda+\rho\A_j^{\top}\left(\sum_{l\neq j}^p\A_l\x^t_l(j) -\frac{\1}{2}\right)_{\!\!+}\right).
  \end{array}
\end{eqnarray*}

We can observe that the condition $\A\x \leq \1$  induces the decomposition of minimizing the augmented Lagrangian function  \eqref{alf} into a set of subproblems with a linear objective function, which makes \eqref{xjko} easier to solve. Therefore, the key point in the derivation of the linearization process is utilizing the fact that the entries in $\A_j\x_j$ are either $1$ or $0$. We next verify this condition is always true in each iteration of the BCD method under the following assumption.
\begin{assumption}\label{cAj}
  Denote by $T_{\A_j}$ the indices of columns of the matrix $\A_j$ containing only zeros and $c_{j_s}$ the $s$-th element of $\c_j$. At least one of the following holds.
  \begin{itemize}
    \item[(i)] For all $j \in \N_p$, there exists $\bx_j \in \X_j$ such that $\A_j\bx_j \leq \textbf{\rm \1}$ and $\{s \in \N_{n_j}: c_{j_s} \ne 0\}\subseteq T_{\A_j}$.
    \item[(ii)] For all $j \in \N_p$, there is at most one nonzero element in the vector $\c_j$ and $\{\bx_j\in \R^{n_j}: \A_j\bx_j \leq \textbf{\rm \1}\} \subseteq \{\bx_j\in \R^{n_j}: \bx_j\in \X_j\}$.
  \end{itemize}
\end{assumption}
\textbf{Remark:}
(a) Assumption \ref{cAj} (i) requires that if an element $\bx_{j_s}$ satisfies $\bx_{j_s}=1$, then at least one of $c_{j_s}$ and the $s$-th column vector of $\A_j$ is zero. Assumption \ref{cAj} (ii) requires that if the block constraint set $\X_j$ is large enough such that  $\A_j\bx_j \leq \textbf{\rm \1}$ holds, then the weight $\c_j$ corresponding to each block variable has only one nonzero element.
(b) Assumption \ref{cAj} usually holds when $\A$ and $\c$ are sparse. 
This is particularly true in the train timetabling problem with the objective of scheduling more trains, as will be shown in Section \ref{sec_app}.


\begin{lemma}\label{Axt1}
  Suppose Assumption \ref{cAj} hold. From any starting point $\bx^0$, the solution generated by the BCD method with the classical update at each iteration satisfies $\A_j\bx_j^{t+1} \leq \textbf{\rm\1}$ for all $j \in \N_p$ and $t=0,1,2,...$.
\end{lemma}
\begin{proof}
Let $\x_j^{t+1}$ be the solution generated by the BCD method after $t$-th classical update.
We argue by contradiction and suppose that there exists  $i \in \N_m$ such that $\A_{i,j}\x_j^{t+1}>1$. Then for any $\bar{\x}_j \in \X_j$ satisfying $\A_j\bar{\x}_j \leq 1$, we have
$$L^t_j(\x_j^{t+1},\lambda,\rho) < L^t_j(\bar{\x}_j,\lambda,\rho),~~\forall j \in \N_p,$$
which implies that for any $ j \in \N_p$,
$$\left(\c_j+\A_j^{\top}\lambda+\rho\A_j^{\top}\left(\sum\limits_{l\neq j}^p\A_l\x^t_l(j) -\frac{\1}{2}\right)_{\!\!+}\right)^{\top} (\x_j^{t+1}-\bar{\x}_j) <   0,$$
then
\begin{eqnarray}\label{cont}
\begin{array}{ll}
  &\c_j^{\top} (\x_j^{t+1}-\bar{\x}_j)+\lambda^{\top} (\A_j\x_j^{t+1}-\A_j\bar{\x}_j) +\\ &\rho\left(\sum_{l\neq j}^p\A_l\x^t_l(j) -\frac{\1}{2}\right)_{\!\!+}^{\top} (\A_j\x_j^{t+1}-\A_j\bar{\x}_j) <   0.
  \end{array}
\end{eqnarray}

We consider following two cases corresponding to Assumption \ref{cAj}:

Case 1: Assumption \ref{cAj} (i) implies  $c_{j_s}=0$ for any $s \in S_j:=\{s \in \N_{n_j}:\x_{j_s}^{t+1}=1, \A_{i,j_s}=1\}$. Taking $\bar{\x}_j=(\x_{j_1}^{t+1},...,\x_{j_s}^{t+1}-1,...,\x_{j_{n_j}}^{t+1})$ with $s \in S_j$ such that $\A_j\bar{\x}_j \leq \1$ gives us
\begin{equation}\label{cxi}
\hspace{-2mm}  \c_j^{\top}(\x_j^{t+1}-\bar{\x}_j)=\sum\nolimits_s c_{j_s}=0~ \text{and}~ \A_{i,j}\x_j^{t+1}-\A_{i,j}\bar{\x}_j \geq 1.
\end{equation}
This contradicts \eqref{cont} since $\lambda\geq 0, \rho > 0$.

Case 2: Assumption \ref{cAj} (ii) implies that there exists
\begin{equation}\label{xjii}
  \bar{\x}_j=(\x_{j_1}^{t+1},...,\bar{\x}_{j_{s'}},...,\x_{j_s}^{t+1}-1,...,\x_{j_{n_j}}^{t+1}) ~\text{with}~\bar{\x}_{j_{s'}}=1
\end{equation}
for any $s \in S_j$ and $s' \neq s $ such that $\A_j\bar{\x}_j \leq \1$. If $c_{j_{s'}} \geq 0$, we can apply the same procedure in Case 1 and arrive at the contradiction. If $c_{j_{s'}}<0$, we take $\bar{\x}_j$ in \eqref{xjii}, then $\x_{j_{s'}}^{t+1}-\bar{\x}_{j_{s'}} \leq 0$. Hence,
\begin{eqnarray*}
  \c_j^{\top}(\x_j^{t+1}-\bar{\x}_j)=c_{j_{s'}}(\x_{j_{s'}}^{t+1}-\bar{\x}_{j_{s'}}) \geq 0, \
  \A_{i,j}\x_j^{t+1}-\A_{i,j}\bar{\x}_j \geq 1,
\end{eqnarray*}
which is a contradiction that completes the proof.
\qed
\end{proof}

Overall, we can reformulate the subproblem \eqref{xjko} into \eqref{xjkot} based on the special structure of the model \eqref{bLP} under certain conditions, and thus make it easier to solve. It is worth noting that if Assumption \ref{cAj} is not satisfied, we can consider \eqref{xjkot} as a linear approximation of \eqref{xjko}.
This linearization technique is widely used, including in works \cite{yao2019admm} and \cite{ZHANG2020102823}. However, they lacked theoretical guarantees. Our main result shows that under specified assumptions, using the update \eqref{xjkot} in the ALM-C method is equivalent to solving \eqref{xjko} exactly, theoretically ensuring the effectiveness of our method.

\subsection{Finding a good feasible solution by set packing}
Since the BCD method may not always yield a feasible solution\cite{gonzalez2011heuristic}, we can adopt several refinement strategies that are very useful to find a feasible solution to the problem \eqref{bLP} in practice.  
Very often, \eqref{bLP} represents a problem where we  optimize under limited resources, and this provides us a view of set packing problems. Since we produce candidate solutions all along, a natural idea is to utilize past iterates to construct a feasible solution. A simple strategy could be constructing a solution pool for each $j$ that includes the past BCD iterates $V^k_j = \{\x^1_j,\x^2_j,...,\x^k_j\}, j = 1,..., p$. Intuitively, the solution pool may include feasible or ``nice'' solutions in some sense.

To illustrate the above approach, we first introduce the iterative sequential technique.
\subsubsection{A sweeping technique}
The most simple technique is probably to select a subset of blocks, one by one, until the infeasibility is detected. We present this sequential method in Algorithm \ref{simpleseq}, which can be understood by simply sweeping the blocks and selecting a candidate solution from $V_j^k$ if feasibility is still guaranteed, otherwise simply skip the current block and continue.

\begin{algorithm}
  \caption{A sweeping technique}
  \label{simpleseq}
  \KwIn{The set of past BCD solutions $V^k_j, j = 1,..., p$.}
  \KwOut{A feasible solution $\hat \x^k$}
  \While{the termination is not satisfied}
  {\For{$j=1,2,...,p$}
    {Select $\v_j \in V_j^k$ \;
      \eIf{$\A_j \textbf{\rm \v}_j +\sum\limits_{l=1}^{j-1} \A_l\hat \bx^k_l - \b \le 0$}
      { let $\hat \x_j^{k}=\v_j$,}{
        let $\hat \x_j^{k}=0$\;}
    }
  }
\end{algorithm}

\subsubsection{A packing technique}
Let us further explore the idea of selecting solutions in a systematic way.
Formally, for each block $j \in \N_p$, we introduce a set of binary variables $\l_j$ that defines the current selection:
$$
  \hat \x_j^k = X_j^k\l_j,~ \l_j \in \{0, 1\}^{k},
$$
where $X_j^k=[\x^1_j;\x^2_j;...;\x^k_j]$. We consider the following problem:
\begin{subequations}\label{mis_model}
  \begin{align}
    \min_{\l_j \in \{0, 1\}^{k}} ~    & ~ \sum\nolimits_{j=1}^p c_j^\top X_j^k \l_j             \\
    \label{eq.matching}\text{s.t.}~~~ & ~ \l_j^\top \1 \le 1,~\forall j=1,...,p, \\
    \label{eq.knapsack}~              & ~ \sum\nolimits_{j=1}^p \A_jX_j^k \l_j \le \b.
  \end{align}
\end{subequations}
In view of \eqref{eq.matching}, one may recognize the above problem as a restricted master problem appearing in column generation algorithms where \eqref{eq.knapsack} stands for a set of knapsack constraints.

Specifically,  the coupling constraints \eqref{eq.knapsack} in our model are cliques, representing complete subgraphs where each pair of distinct nodes is connected. This allows us to reformulate the model as a maximum independent set problem.
Therefore,  we can find a feasible solution $\x^*$ which satisfies $\A\x^* \leq \b$ by solving the problem \eqref{mis_model}. 
 Since this problem is still hard, we only solve the relaxation problem of the maximal independent set. Now we go into details. Since the knapsack \eqref{eq.knapsack} means the candidate solutions may conflict, we can construct a conflict graph $F=(V,E)$. In this graph, $V$ represents the set of nodes, where each node corresponds to a solution generated as the algorithm proceeds, and $E$ is the set of edges that connect two conflicting solutions, meaning they violate the coupling constraints.
In this view, we only have to maintain the graph $F$ and find a maximal independent set $\mathcal{K}$. Therefore, the output feasible point $\x_j^*$ corresponds to $\v_j \in \mathcal{K}$ for each $j \in \N_p$. If $\v_j \notin \mathcal{K}$, then $\x_j^*=0$. We summarize the maximal independent set technique in Algorithm \ref{MIS}. We also note that Algorithm \ref{simpleseq} can be seen as a special case of Algorithm \ref{MIS}.

\begin{algorithm}
  \caption{A packing (maximal independent set) technique}
  \label{MIS}
  \KwIn{The BCD solution $\x^k$, last conflict graph \(F^{k-1} = (V^{k-1}, E^{k-1})\)}
  \KwOut{A feasible solution $\x_*^k$}
  \textbf{Step 1: Conflict graph update}\\
  \For{ $j =1,2,...,p$}{
    Collect new candidate paths \(\tilde V^{k}_j\) for block $j$\;
    (\textbf{1.1 node-update}) \(V^k \leftarrow V^{k-1} + \tilde V^{k}_j\)\;
    (\textbf{1.2 self-check}) \(E^k \leftarrow E^{k-1} + \{(p, p') \mid \forall p\neq p', p \in V^{k}_j, p'\in V^{k}_j\}\)\;
      (\textbf{1.3 edge-completion}) $E^k \leftarrow E^k +  \Big\{(p, p')\mid$
       $\left. \text{if}~ p, p' \text{are compatible}~\text{for}~\forall p \in  \tilde{V}^{k}_j, p' \in V^k \backslash V^k_j \right\}$,  here `compatible' means that these two nodes (block variables) satisfy the binding constraints;
  }
  \textbf{Step 2: Maximal independent set (MIS) for a feasible solution}\\
  Select a candidate solution set \(K \subset V^k\)\;
  \For{\(v \in K\)}{
    Compute a maximal independent set  \(\mathcal K(v) \) with respect to \(v\) in \(O(|E^k|)\) iterations\;
  }
  Compute \(\x^* = \arg \max_v \mathcal K(v)\).
\end{algorithm}

Based on the above-mentioned techniques, we improve the ALM and propose a customized ALM  in Algorithm \ref{PBALM}.
\begin{algorithm}[htbp]
  \caption{A customized ALM}
  \label{PBALM}
  \KwIn{ $\x^0,\lambda^0,\rho^0>0$ and the best objective function value $f^* = +\infty$. Set $k=0$.}
    \KwOut{A local (global) optimal solution $ \x^*$.}
  \While{the termination is not satisfied}
  {\textbf{Step 2: Construct solution using Alg. 1}\\
    Update the BCD solution $\x^{k+1}$ by the procedure in lines 2-5 of Alg. \ref{ALM}\;
    \textbf{Step 3: Generate a feasible solution}\\
    \eIf{the BCD solution is not feasible}
    {transform the BCD solution to a feasible solution $\bar{\x}^{k+1}$ by calling a refinement method in Alg. \ref{simpleseq} or \ref{MIS}\;} {let $\bar{\x}^{k+1}=\x^{k+1}$\; }
    \textbf{Step 4:  Update the best solution}\\
    \If{$\c^{\top}\bar{\x}^{k+1} \leq f^*$}
    {Set $f^*=\c^{\top}\bar{\x}^{k+1}$ and $\x^*= \bar{\x}^{k+1}$;}
     Update the Lagrangian multipliers  $\lambda^{k+1}$  by \eqref{lam} and the penalty coefficient $\rho^{k+1}$ by \eqref{rho}. Let $k = k+1$.}
\end{algorithm}

Although Algorithms \ref{simpleseq} and \ref{MIS} can help us find a feasible solution to problem \eqref{bLP}, the quality of output by them remains unjustified.
To evaluate and improve the quality of the solution, the simple way is to estimate the upper and lower bounds of the objective function value, and then calculate the gap between them. The smaller the gap, the better the current solution. Obviously, our method can provide an upper bound of the objective function value. As for the generation of the lower bound, we can use the following method. (i) LP relaxation: by directly relaxing the binary integer variables to $[0,1]$ continuous variables, we solve a linear programming problem exactly to obtain the lower bound of the objective function value. (ii) LR relaxation: since the Lagrangian dual problem \eqref{LR} is separable, we can solve the decomposed subproblems exactly to obtain the lower bound of the objective function value.

In general, the Lagrangian dual bound is at least as tight as the linear programming bound obtained from the usual linear programming relaxation \cite{conforti2014integer}. Note that the relaxation must be solved to optimality to yield a valid bound. Therefore, we can use a combination of LR method and Alg. \ref{PBALM}. To be specific, LR aims at generating the lower bound of the objective function value in \eqref{bLP}, the solution is usually infeasible. Steps 2 and 3 in Alg. \ref{PBALM} are used for generating feasible solutions of \eqref{bLP}. The smaller the gap between the lower bound and the upper bound, the closer the feasible solution is to the global optimal solution of the problem. It can be seen from the iterative procedure that after many iterations, this algorithm can generate many feasible solutions, and the lower bound of the model is constantly improving. Finally, we select the best solution. The combination of these two methods takes advantage of the ALM and the LR method, it not only finds a good feasible solution but also evaluates the quality of the solution. 

Compared with the ADMM-based method in \cite{ZHANG2020102823}, we utilize BCD method to perform multiple iterations to solve the subproblem \eqref{xsub} until the solutions remain unchanged, which can improve the solution accuracy of the subproblem, thereby reducing the total number of iterations. Moreover, we adopt different refinement techniques to further enhance the solution quality.


\section{Convergence Analysis}\label{sec_con}
In this section, we present the convergence analysis of the block coordinate descent method for solving the augmented Lagrangian relaxation problem \eqref{df} and the augmented Lagrangian method for solving the dual problem \eqref{zdp}.  Unless otherwise stated, the convergence results presented in this section do not rely on Assumptions \ref{Ab1} and \ref{cAj}.
\subsection{Convergence of BCD}
We begin this section with the property of augmented Lagrangian function \eqref{alf} that is fundamental in the convergence analyis.
\begin{pro}\label{lipschitz}
  The gradient of $L(\bx,\lambda,\rho)$ at $\bx$  is Lipschitz continuous with constant $\kappa$ on $\X$, namely,
  $$\|\nabla L(\bx,\lambda,\rho)-\nabla  L(\bar{\bx},\lambda,\rho)\|\leq \kappa\|\bx-\bar{\bx}\|$$
  for all  $\bx, \bar{\bx}\in \X$, where $\kappa=\rho\|\A\|_2^2$. Furthermore,
  \begin{equation}\label{smooth}
    L(\bar{\bx},\lambda,\rho) \leq L(\bx,\lambda,\rho) + \langle \bar{\bx}-\bx, \nabla L(\bx,\lambda,\rho) \rangle+\frac{\kappa}{2}\|\bar{\bx}-\bx\|^2
  \end{equation}
  for all $\bx, \bar{\bx}\in \X$.
\end{pro}
\begin{proof}
For any $\bx, \bar{\bx}\in \X$,
\begin{align*}
 & \|\nabla L(\bx,\lambda,\rho)-\nabla L(\bar{\bx},\lambda,\rho)\|\\
  =&\left\|\rho\A^{\top}(\A\x-\b)_+-\rho\A^{\top}(\A\bar{\x}-\b)_+\right\|\\
  \leq&\rho\|\A \|_2\left\|(\A\x-\b)_+ -(\A\bar{\x}-\b)_+\right\|\\
  \leq&\rho\|\A \|_2\left\|\A\x-\A\bar{\x}\right\| \leq \rho\|\A \|^2_2\left\|\x-\bar{\x}\right\|.
\end{align*}
Then we have
\begin{align*}
  &\langle\nabla L(\bx,\lambda,\rho)-\nabla L(\bar{\bx},\lambda,\rho),\bx-\bar{\bx} \rangle\\
  \leq& \|\nabla L(\bx,\lambda,\rho)-\nabla L(\bar{\bx},\lambda,\rho)\|\|\bx-\bar{\bx}\| \leq \kappa\|\bx-\bar{\bx}\|^2.
\end{align*}
It follows from the convexity of the function $f(\x)=\frac{\kappa}{2}\|\x\|^2-L(\bx,\lambda,\rho)$ that
$$f(\bar{\bx}) \geq f(\x)+\nabla f(\x)^{\top}(\bx-\bar{\bx}),$$
which implies the estimate \eqref{smooth}.
\qed
\end{proof}

Considering two different updates in the BCD method, we first summarize the convergence property of the BCD method for solving \eqref{xjko}, which has been presented in \cite{jager2020blockwise}. Before that, we give the definition of the blockwise optimal solution, which is also called coordinatewise minimum point in \cite{tseng2001convergence}. Given parameters $\lambda$ and $\rho$, a feasible solution $\x^*$ is called a blockwise optimal solution of the problem \eqref{df}  if for each $j \in \N_p$, we have for all $\x=(\x_1^*,...,\x_{j-1}^*,\x_j,\x_{j+1}^*,...,\x_p^*)\in \X$,
$L(\x^*,\lambda,\rho) \leq L(\x,\lambda,\rho).$
\begin{lemma}\label{bos} Suppose Assumptions \ref{Ab1} and \ref{cAj} hold. If the starting point satisfies $\bx^0 \in \X$, then the BCD method for solving \eqref{xjko} is always executable and terminates after a finite number of iterations with a blockwise optimal solution of the problem \eqref{df}.
\end{lemma}
\begin{proof}
If Assumptions \ref{Ab1} and \ref{cAj} hold,  Lemma \ref{Axt1} tells us that the subproblem \eqref{xjko} can be solved exactly. (i) Since the constraint set of each subproblem \eqref{xjko} is bounded, then all subproblems have an optimal solution. Therefore, the BCD method for solving \eqref{xjko} is executable. (ii) The result in (i) illustrates that the sequence $\{\x^t\}$ generated by the BCD method exists. Thus the sequence of objective function values $\{L(\bx^{t},\lambda,\rho)\}$ can only take finitely many different values. This together with the monotonically decreasing property of the function $L(\bx,\lambda,\rho)$ yields that $L(\bx^{t},\lambda,\rho)$ must become a constant. Then the BCD method is executable. (iii) By the definition of the blockwise optimal solution and the fact that the subproblem \eqref{xjko} can be solved exactly, we arrive at the conclusion.
\qed
\end{proof}

A similar conclusion can be found in \cite{jager2020blockwise}, but it does not clarify how to solve the subproblem. Note that each (global) optimal solution of the model \eqref{df} is blockwise optimal, but not vice versa. Subsequently, we analyze the convergence of \eqref{xjkl}. The following lemma ensures decreasing the function value of $L$ after each iteration if $\bx^{t+1}\neq\bx^t$ and the step size is chosen properly.
\begin{lemma}\label{tauka}
  Let $\{\bx^t\}_{t \in \N}$
  be a sequence generated by \eqref{xjkl}, we obtain
  \begin{equation}\label{lx}
    (\frac{1}{2\tau}-\frac{\kappa}{2})\|\bx^{t+1}-\bx^t\|^2 \leq L(\bx^{t}, \lambda,\rho)-L(\bx^{t+1},\lambda,\rho).
  \end{equation}
\end{lemma}
\begin{proof}
Since $$\x_j^{t+1} =\mathcal{T}_{\X_j}\left(\tau g_j(\x^t)^{\top}\x_j+\frac{\textbf{1}}{2}-\x_j^t\right)$$ and $\x_j^t \in \X_j$  for all $j \in \{1,2,...,p\}$, we have
$$\left(\tau g_j(\x^t)+\frac{\textbf{1}}{2}-\x_j^t\right)^{\!\!\top}\!\x_j^{t+1} \leq \left(\tau g_j(\x^t)+\frac{\textbf{1}}{2}-\x_j^t\right)^{\!\!\top}\!\x_j^{t},$$ which implies that
\begin{equation}\label{l2e1}
  2\tau\langle \x_j^{t+1}-\x_j^t, g_j(\x^t) \rangle+\|\x_j^{t+1}-\x_j^t\|^2 \leq 0.
\end{equation}
Proposition \ref{lipschitz} tells us that
\begin{align}\label{l2e2}
 & L(\x^{t+1},\lambda,\rho)-L(\x^{t},\lambda,\rho)\nonumber\\
  \leq &\sum\nolimits_{j=1}^p\langle \x_j^{t+1}-\x_j^t, g_j(\x^t) \rangle+\frac{\kappa}{2}\sum\nolimits_{j=1}^p\|\x_j^{t+1}-\x_j^t\|^2.
\end{align}
Combining inequalities \eqref{l2e1} and \eqref{l2e2} yields \eqref{lx}. The proof is completed.
\qed
\end{proof}

This lemma tells us that a small step size satisfying $\tau <1/\kappa$ leads to a decrease in the function value $L$ when $\bx^{t+1}\neq\bx^t$. However, the following lemma states that when the step size is too small, the iteration returns the same result. We let $g(\x)$ denote the gradient of $L(\x,\lambda,\rho)$ at $\x$.
\begin{lemma}\label{stuck}
  If the step size satisfies $0 < \tau < \frac{1}{2\|g(\bx^t)\|}$ when $g(\bx^t) \neq \textbf{0}$, then it holds
  that $$\bx^t =\mathcal{T}_{\X}\left(\tau g(\bx^t)+\frac{\bold{1}}{2}-\bx^t\right).$$
\end{lemma}
\begin{proof}
If $0 < \tau < \frac{1}{2\|g(\x^t)\|}$, for all $\x^t \neq \x \in \X$, we have
\begin{equation*}
  -2\tau g(\x^t)^{\top}(\x^t-\x) \leq 2\tau \|g(\x^t)\|\|\x^t-\x\|\leq \|\x^t-\x\|< \|\x^t-\x\|^2,
\end{equation*}
where the last inequality holds due to $\x^t, \x \in \{0,1\}^n$. It yields that
$$\left(\tau g(\x^t)+\frac{\textbf{1}}{2}-\x^t\right)^{\top}\x^t \leq \left(\tau g(\x^t)+\frac{\textbf{1}}{2}-\x^t\right)^{\top}\x.$$
By the definition of the operator $\mathcal{T}_{\X}(\cdot)$, we arrive at the conclusion.
\qed
\end{proof}

Based on Lemma \ref{stuck}, the implementation of the BCD method heavily relies on the choice of step size $\tau$. Hence, we give a definition of a $\tau$-stationary point.
\begin{definition}
  For the AL relaxation problem \eqref{df}, if a point $\bx^*$ satisfies $$\bx^* =\mathcal{T}_{\X}\left(\tau g(\bx^*)+\frac{\textbf{1}}{2}-\bx^*\right)$$ with $\tau > 0$, then it is called a $\tau$-stationary point.
\end{definition}

For the augmented Lagrangian relaxation problem \eqref{df}, given parameters $\lambda$ and $\rho$, we say that $\x^*$ is a $\delta$-local minimizer if there is an integer $\delta>0$ such that
$$L(\x,\lambda,\rho) \geq L(\x^*,\lambda,\rho),~~\text{for~all}~ \x\in \mathcal{N}(\x^*,\delta) \cap \X.$$
Note that a $n$-local minimizer is a global minimizer due to the fact $\|\x-\x^*\| \leq n$ for all $\x,\x^* \in \X$.
The following important result reveals a relationship between a $\tau$-stationary point and a global minimizer of the problem \eqref{df}.
\begin{theorem}\label{sta_local}
  We have the following relationships between the $\tau$-stationary point and the local minimizer of the problem \eqref{df}.
  \begin{itemize}
    \item[(i)] If $\bx^*$ is a local minimizer, then $\bx^*$ is a $\tau$-stationary point for any step size $0 < \tau <1/\kappa$.
    \item[(ii)] If $\bx^*$ is a $\tau$-stationary point with $\tau >\delta/2$, then  $\bx^*$ is a $\delta$-local minimizer of the problem \eqref{df} under the assumption that the entries in $\A,\b,\c,\lambda$ and $\rho$ are integral.
  \end{itemize}
\end{theorem}
\begin{proof}
(i) If $\x^*$ is a local minimizer, then for any $\x \in \X$ we have
\begin{align*}
  L(\x^*,\lambda,\rho) \leq & L(\x,\lambda,\rho)\\
  \leq& L(\x^*, \lambda, \rho)+\langle\nabla L(\x^*, \lambda, \rho), \x-\x^* \rangle +\frac{\kappa}{2}\|\x-\x^*\|^2.
\end{align*}
Therefore, $$\langle\nabla L(\x^*, \lambda, \rho), \x-\x^* \rangle \geq -\frac{\kappa}{2}\|\x-\x^*\|^2\geq -\frac{1}{2\tau}\|\x-\x^*\|^2,$$
which implies that $\x^*$ is a $\tau$-stationary point.

(ii) Since $\x^*$ is a stationary point, then for any $\x  \in \X$ we have
\begin{equation}\label{stax}
  \left(\tau g(\x^*)+\frac{\textbf{1}}{2}-\x^*\right)^{\top}\x^* \leq \left(\tau g(\x^*)+\frac{\textbf{1}}{2}-\x^*\right)^{\top}\x.
\end{equation}
For any $\x \in \X \cap \mathcal{N}(\x^*,\delta)$, we define the  index sets $\mathbb{J}:=\{l \in \N_n: x^*_l=0, x_l=1\}$ and $\bar{\mathbb{J}}:=\{l \in \N_n: x^*_l=1, x_l=0\}$.  Then
\begin{eqnarray*}
  &&\left(\tau g(\x^*)+\frac{\textbf{1}}{2}-\x^*\right)^{\top}\left(\x-\x^*\right)\\
  &=&\sum\nolimits_{l \in \mathbb{J}}\left(\tau g_l(\x^*)+\frac{1}{2}\right)-\sum\nolimits_{l \in \bar{\mathbb{J}}}\left(\tau g_l(\x^*)-\frac{1}{2}\right)\\
  &\leq&\tau\left(\sum\nolimits_{l \in \mathbb{J}} g_l(\x^*)-\sum\nolimits_{l \in \bar{\mathbb{J}}}g_l(\x^*)\right)+\frac{\delta}{2},
\end{eqnarray*}
which together with \eqref{stax} and $\tau>\frac{\delta}{2}$ yields that
\begin{eqnarray*}\label{ggra}
  \sum\nolimits_{l \in \mathbb{J}} g_l(\x^*)-\sum\nolimits_{l \in \bar{\mathbb{J}}}g_l(\x^*) \geq -\frac{\delta}{2\tau} >-1.
\end{eqnarray*}
Since the entries in $\A,\b,\c,\lambda$ and $\rho$ are integral, then $g_l(\x^*)$ is an integer for every $l \in \N_n$.
This implies that
\begin{eqnarray*}\label{ggra}
  \sum\nolimits_{l \in \mathbb{J}} g_l(\x^*)-\sum\nolimits_{l \in \bar{\mathbb{J}}}g_l(\x^*) \geq 0.
\end{eqnarray*}
Then for any $\x \in \X \cap \mathcal{N}(\x^*,\delta)$, it follows from the convexity of $L$ that
\begin{eqnarray*}
 && L(\x,\lambda,\rho)-L(\x^*,\lambda,\rho)\\
  &\geq & \langle \nabla L(\x^*,\lambda,\rho), \x-\x^*  \rangle\\
  &=& \sum\nolimits_{l \in \mathbb{J}}g_l(\x^*)(x_l-x_l^*)+\sum\nolimits_{l \in \bar{\mathbb{J}}}g_l(\x^*)(x_l-x_l^*) \\&=&\sum\nolimits_{l \in \mathbb{J}} g_l(\x^*)-\sum\nolimits_{l \in \bar{\mathbb{J}}}g_l(\x^*) \geq 0.
\end{eqnarray*}
Therefore, $\x^*$ is a $\delta$-local minimizer of the problem \eqref{df}.
\qed
\end{proof}

We can observe from Theorem \ref{sta_local} that every blockwise optimal solution of \eqref{df} is a $\tau$-stationary point. Conversely, if the entries in $\A,\b,\c,\lambda$ and $\rho$ are integral and $\tau >\frac{1}{2}\max_j\{n_j\}$, then every  $\tau$-stationary point of \eqref{df} is blockwise optimal. These results on relationships between $\tau$-stationary point, blockwise optimal solution and local (global) minimizer are summarized in Figure \ref{relationship} intuitively.

\begin{figure*}[t]
  \usetikzlibrary{math,calc,intersections,through,angles,arrows.meta,shapes.geometric,shadows,quotes,spy,datavisualization,datavisualization.formats.functions,plotmarks}
  \tikzset{every picture/.style={samples=300,smooth,line join=round,thick,>=stealth}}
  \newcommand{\px}{/\!/}
  \centering
  \begin{tikzpicture}[thick,font=\small,scale=0.8]
    \tikzmath{\abx=-0.2;}
    \coordinate(A)at(0,0);
    \coordinate(B)at(10,0);
    \coordinate(Ab)at(0,\abx);
    \coordinate(Bb)at(10,\abx);
    \coordinate[label=below:{{ $\delta$-local minimizer}}](XX)at(5,-2.5);
    \draw[-stealth]($(A)+(0,0.1)$)node[left]{ $\tau$-stationary point}--node[above,pos=0.55]{
    Condition (a), $\tau >\frac{1}{2}\max_j\{n_j\}$}
    ($(B)+(0,0.1)$)node[right]{{blockwise optimal solution}};
    \draw[-stealth] ($(B)+(0,-0.1)$)--node[below=-0.2cm]{}($(A)+(0,-0.1)$);

    \draw[-stealth]($(A)-(0.6,0.5)$)--node[above,sloped,pos=0.55]{Condition (a), $\tau >\delta/2$}
    ($(XX)-(0.2,-0.1)$);
    \draw[-stealth]($(XX)-(0.5,0)$)--node[below,sloped,pos=0.55]{$0 < \tau <1/\kappa$}
    ($(A)-(0.7,0.7)$);

    \draw[-stealth]($(B)+(1,-0.5)$)--node[above,sloped,pos=0.55]{Assumption \ref{Ab1}}
    ($(XX)-(-0.2,-0.1)$);
    \draw[-stealth]($(XX)-(-0.5,0)$)--node[below,sloped,pos=0.55]{$\delta=n$}
    ($(B)+(1.1,-0.7)$);
  \end{tikzpicture}
  \caption{The relationships among $\tau$-stationary point, blockwise optimal solution and local minimizer. Condition(a): the entries in $\A,\b,\c,\lambda$ and $\rho$ are integral.}\label{relationship}
\end{figure*}
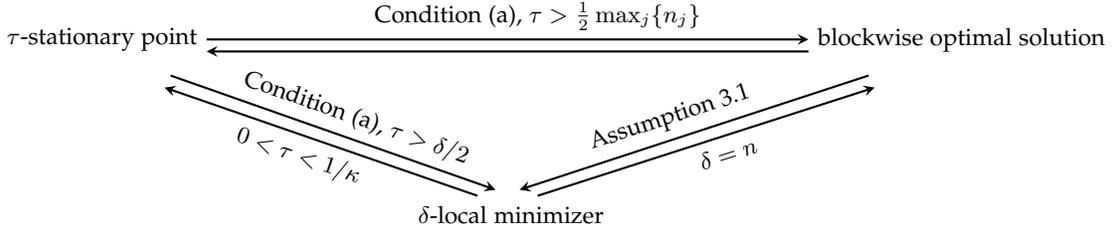

We finally present the main result on the convergence of the BCD method for solving \eqref{xjkl}.
\begin{theorem}\label{cp}
  (Convergence properties) Let $\{\bx^t\}_{t \in \N}$
  be a sequence generated by \eqref{xjkl}. If the step size  satisfies $0 < \tau <\frac{1}{2\kappa}$, then we have
  \begin{itemize}
    \item[(i)] The sequence $\{L(\bx^t,\lambda,\rho)\}_{t \in \N}$ is nonincreasing and
      \begin{equation}\label{des}
        \frac{\kappa}{2}\|\bx^{t+1}-\bx^t\|^2 \leq L(\bx^t, \lambda,\rho)-L(\bx^{t+1},\lambda,\rho).
      \end{equation}
    \item[(ii)]  The sequence $\{\bx^t\}$ converges to a $\tau$-stationary  point after at most $\lceil  \frac{2C\sqrt{n} +\kappa n}{\kappa} \rceil$ iterations, where $C=\max_{\bx \in \X}\| \nabla L(\bx, \lambda, \rho)\|$.
  \end{itemize}
\end{theorem}
\begin{proof}
(i) It is obvious from \eqref{lx} that if $0 < \tau <\frac{1}{2\kappa}$, then \eqref{des} holds.

(ii) Given the parameters $\lambda \in \R^m_+$ and $\rho>0$, we can observe that the model \eqref{df} can be equivalently written as
\begin{equation*}\label{dg_c}
  \min  F(\x):= L(\x,\lambda,\rho) +\delta_{\X}(\x),
\end{equation*}
where $\delta_{\X}(\x)$ is the indicator function of the set $\X$, i.e., $\delta_{\X}(\x)=
  +\infty,  \text{if}~ \x \in \X$ and $0$ otherwise. To verify that the sequence $\{\bx^t\}$ converges to a $\tau$-stationary  point, we examine that the conditions in \cite[Theorem 1]{bolte2014proximal} hold.

Firstly, since the set $\X=\{\x \in \{0,1\}^n:\B\x \leq \d\}$ is semi-algebraic, then $\delta_{\X}(\x)$ is a Kurdyka-\L{}ojasiewicz  (KL) function \cite{attouch2013convergence,bolte2014proximal}. It is obvious to see that $L(\x,\lambda,\rho)$ is also a KL function, and hence the function $F(\x)$ satisfies the KL property, which is a crucial condition ensuring convergence.

Secondly, the Lipschitz constant $\kappa>0$ in Proposition \ref{lipschitz} is bounded if $\rho$ has an upper bound, and the problem \eqref{df} is inf-bounded.

Thirdly, for any $\x, \bar{\x}\in \X$ with $\bar{\x}=(\x_1,...,\bar{\x}_j,...,\x_p)$,
\begin{eqnarray*}
  &&\|\nabla_{\x_j} L(\x,\lambda,\rho)-\nabla_{\x_j} L(\bar{\x},\lambda,\rho)\|\\
  &=&\left\|\rho\A_j^{\top}(\A\x-\b)_+-\rho\A_j^{\top}(\A\bar{\x}-\b)_+\right\|\\
  &\leq&\rho\|\A_j \|_2\left\|(\A\x-\b)_+ -(\A\bar{\x}-\b)_+\right\|\\
  &\leq&\rho\|\A_j \|_2\left\|\A\x-\A\bar{\x}\right\| \leq \rho\|\A_j \|^2_2\left\|\x_j-\bar{\x}_j\right\|.
\end{eqnarray*}
Therefore, the result follows from \cite[Theorem 1]{bolte2014proximal}.

Suppose $\x^*$ is a $\tau$-stationary point, then for every $\x \in\X$, we have
\begin{align}\label{t1e1}
  &L(\x, \lambda, \rho)-L(\x^*, \lambda, \rho)\nonumber\\
  \leq& \langle \nabla L(\x^*, \lambda, \rho), \x-\x^* \rangle +\frac{\kappa}{2}\|\x-\x^*\|^2 \nonumber\\
  \leq& C\|\x-\x^*\| +\frac{\kappa}{2}\|\x-\x^*\|^2\leq C\sqrt{n} +\frac{\kappa n}{2}.
\end{align}
If each iteration always finds a new point until arriving at $\x^*$, we get
\begin{align}\label{t1e2}
  \frac{\kappa T}{2}&\leq \sum\nolimits_{t=0}^T \left(L(\bx^t, \lambda,\rho)-L(\x^{t+1},\lambda,\rho)\right)
 \nonumber  \\
  &\leq L(\x^0, \lambda, \rho)-L(\x^*, \lambda, \rho),
\end{align}
where the first inequality holds due to the conclusion in (i) and the fact $\|\x^{t+1}-\x^t\| \geq 1$. Thus combining \eqref{t1e1} with \eqref{t1e2} yields that
$T \leq \lceil \frac{2C\sqrt{n} +\kappa n}{\kappa} \rceil$. It implies that we
only need at most $\lceil \frac{2C\sqrt{n} +\kappa n}{\kappa} \rceil$ steps to arrive at a $\tau$-stationary point.
\qed
\end{proof}

If each iteration of the BCD method always finds a new point, then according to Theorem \ref{cp} (i), we have $\sum_{t=0}^{\infty}\|\bx^{t+1}-\bx^t\|^2 < +\infty$. By Theorem 1 in \cite{bolte2014proximal}, we can conclude that
$\lim_{t \rightarrow \infty}  \x^t = \x^*,$ where $\x^*$ is a limit point of the sequence $\{\x^t\}_{t \in \N}$.  If we choose an initial point that is already close to an optimal solution and an appropriate step size $\tau$, based on specific convergence criteria detailed in Theorem \ref{cp}, the BCD method can find a global solution of problem \eqref{df}.

\subsection{Convergence of ALM}
Assume the BCD method returns a global minimizer $\x^*$  to the augmented Lagrangian relaxation problem (\ref{df}) in each inner loop. In this subsection, we focus on the convergence property of the projected subgradient method for solving the dual problem \eqref{zdp}. 
%
%
%
For convenience, we introduce the following constants for subsequent analysis:
\begin{eqnarray*}
  S &:=& {\text{arg}\max}_{\lambda \in \R_+^m,\rho>0}~ d(\lambda,\rho),\\
  \theta &:=& {\min}_{(\lambda,\rho) \in S}~ \|\lambda^0-\lambda\|^2 +(\rho^0-\rho)^2.
\end{eqnarray*}
Let  $d_g^k$ denote the subgradient of $d(\lambda^k,\rho^k)$.
We estimate the distance between the dual function value in each iteration and the optimal value of the problem \eqref{bLP} in the following theorem.
\begin{theorem}\label{alm_converge}
  If we take the step size  $\alpha^k=\frac{\beta_k}{\|d_g^k\|}$ with $\beta_k=\sqrt{\frac{\theta}{K}}$, then,
  $$f^{\text{IP}}-\max_{k \in \{1,2,...,K\}} d(\lambda^k,\rho^k) \leq \frac{\zeta}{2}\sqrt{\frac{5\theta}{K}},$$
  where $\zeta= {\max}_{\bx \in \X}~ \|\A\bx-\b\|^4$. Furthermore, if
  the positive sequence $\{\beta_k\}_{k\in \N}$ is bounded and $\sum_{k \in \N}\beta_k^2 < +\infty.$ Then $(\lambda^k, \rho^k)$ converges to some $(\lambda^*, \rho^*) \in S$.
\end{theorem}
\begin{proof}
Let  $\{(\lambda^k, \rho^k)\}_{k\in \N}$ be the sequence generated by Algorithm \ref{ALM}. We derive from the Lemma 3 in \cite{sun2021decomposition} that
\begin{equation}\label{dgk}
  f^{\text{IP}}-\max_{k \in \{1,2,...,K\}} d(\lambda^k,\rho^k) \leq \frac{\theta+\sum_{k=1}^K(\alpha^k\|d_g^k\|)^2}{2\sum_{k=1}^K \alpha^k}.\end{equation}
For all $k \in \N$, we have
$$\|d_g^k\|^2=\|\A\x^k-\b\|^2+\frac{1}{4}\|(\A\x^k-\b)_+\|^4 \leq \frac{5}{4}\|\A\x^k-\b\|^4\leq \frac{5}{4} \zeta.$$
Then the right-hand side of \eqref{dgk} satisfies
$$\frac{\theta+\sum_{k=1}^K(\alpha^k\|d_g^k\|)^2}{2\sum_{k=1}^K \alpha^k}=\frac{\sqrt{5}\zeta(\theta+\sum_{k=1}^K\beta_k^2)}{4\sum_{k=1}^K \beta_k}= \frac{\zeta}{2}\sqrt{\frac{5\theta}{K}}.$$
Thus we arrive at the first result.

The second claim is then proved in \cite[Theorem 7.4]{2011Nonlinear}, where the proof for the subgradient method is readily extended to the projected subgradient method.
\qed
\end{proof}

Although the theoretical analysis of ALM-C relies on Assumption 3.1, our tests show that ALM-C is also effective in more general settings. Furthermore, we present the ALM-P method as a versatile alternative not relying on this assumption.



\section{Applications}\label{sec_app}
In this section, we show the performance of the proposed algorithms on two practical problems. One is the train timetabling problem and the other is the vehicle routing problem. To show the performance of compared methods, we define three termination criteria: a maximum number of iterations, a time limit, and an optimality gap based on the difference between the objective value generated by our methods and the best known value.  All instances are tested on a MacBook Pro 2019 with 8GB of memory and Intel Core i5 (Turbo Boost up to 4.1 GHz) with 128 MB eDRAM.

\subsection{Capacitated vehicle routing problem}
We consider the capacitated vehicle routing problem with time windows (CVRPTW). The problem is defined on a complete directed graph $G=(V, E)$, where $V=\{0,1,...,n\}$ is the node set and $E$ is the edge set. Node 0 represents the depot where the vehicles are based, and nodes 1 to $n$ represent the customers that need to be served. Each edge $(s, t)$ in $E$ has an associated travel time $T_{st}$. Each customer $s$ has a demand $c_s$ and a service time window $[a_s, b_s]$. We let $d_{st}$ be the distance from node $s$ to node $t$ and $M$ be a large constant. The objective is to construct a set of least-cost vehicle routes starting and ending at the depot, such that each customer is visited exactly once within their time window, and the total demand of customers served in each route does not exceed vehicle capacity $C$.

To formulate this problem as an integer program, we define the following decision variables: (i) $x_{st}^j$: binary variable equal to 1 if edge $(s,t)$ is used by vehicle $j$, 0 otherwise. (ii) $w_s^j$: continuous variable indicating the start of service time at customer $s$ by vehicle $j$. Then the block structured integer linear programming formulation is:
\begin{subequations}\label{CVRP}
\begin{align}
\text{min} &\sum\nolimits_{j\in \N_p}\sum\nolimits_{(s,t)\in E} d_{st} x_{st}^j \label{cvrp-obj}\\
\text{s.t.} &\sum\nolimits_{j\in \N_p} \sum\nolimits_{t \in V:t\neq s} x_{st}^j = 1, \hspace{1.6cm} s\in V\backslash{0} \label{cvrp-degree}\\
&\sum\nolimits_{t \in V\backslash s} x_{st}^j = \sum\nolimits_{t \in V\backslash s} x_{ts}^j, \hspace{6.5mm} i\in V,~j\in \N_p\label{cvrp-flow}\\
&\sum\nolimits_{t\in V \backslash 0} x_{0t}^j = 1, \hspace{3.3cm} j\in \N_p \label{cvrp-start}\\
&\sum\nolimits_{s\in V}\sum\nolimits_{t \in V\backslash s} c_s x_{st}^j \leq C, \hspace{1.85cm} j\in \N_p \label{cvrp-capacity}\\
&w_s^j + T_{st} - M(1-x_{st}^j) \leq w_t^j, (s,t)\in E, j\in \N_p \label{cvrp-time}\\
&a_s \leq w_s^j \leq b_s, \hspace{2.85cm} s\in V,~j\in \N_p \label{cvrp-tw}\\
&x_{st}^j \in {0,1}, \hspace{2.85cm} (s, t)\in E, ~ j\in \N_p \label{cvrp-binary}
\end{align}
\end{subequations}

The block structure lies in the routing variables $x_{st}^j$ for each vehicle $j$, which are constrained by the flow balance, capacity and time window constraints. By relaxing the coupling constraints \eqref{cvrp-degree}, the problem decomposes into separate routing subproblems per vehicle.

\subsubsection{Parameter Setting}
We let ALM-C and ALM-P denote Algorithm \ref{PBALM} using the updates \eqref{xjko} and \eqref{xjkl} in step 2, respectively. We compare our proposed methods with the Gurobi solver (version 11.0.0) and OR-tools on all the instances from the Solomon dataset \cite{solomon1987algorithms}. Since the ADMM in \cite{yao2019admm} lacks adaptability for all instances, then we do not compare with it. To illustrate the scale of these instances, we present a subset of representative examples in Table \ref{mpnsize_CVRP_TW}. The notations $|F|, |V|$ and $|E|$ represent the number of vehicles, the number of customers and the number of edges. $n$ and $n_w$ denote the number of variables $x$ and $w$, respectively. To evaluate the robustness of the compared methods, we conduct experiments on the C1-type instances with various problem sizes, as detailed in Table \ref{tab:solomon1}. The results of all remaining instances from the Solomon dataset are presented in Table \ref{tab:solomon2}.


 We adopt a ``vehicle and route-based" formulation to model the CVRPTW problem, which is different from the space-time network flow formulation used in \cite{yao2019admm}. Therefore, the subproblem is a route problem with capacity and time window constraints. We utilize the Gurobi solver to solve the subproblem for convenience.  Note that the formulation does not affect Gurobi's performance. To ensure a fair comparison and maximize Gurobi's utilization, we have implemented efficient callback functions based on Danzig's formulation to fine-tune its performance. We set a time limit of 2000 seconds for the compared methods, except for or-tools where the time limit is set to 500 seconds,  and use the symbol ``--" to signify that the solver failed to find a feasible solution within the allotted time limit.

\renewcommand{\arraystretch}{1}
\begin{table}[H]
  \caption{A subset of representative examples of the Solomon datasets}
  \label{mpnsize_CVRP_TW}
    \centering
  \setlength{\tabcolsep}{1.5mm}{\begin{tabular}{cccccccc}
      \toprule
      \multirow{1}{*}{No.} & \multirow{1}{*}{$|F|$} & \multirow{1}{*}{$|V|$} & \multirow{1}{*}{$|E|$} & \multirow{1}{*}{$m$} & \multirow{1}{*}{$q$} & \multirow{1}{*}{$n$} & \multirow{1}{*}{$n_w$}\cr
      \midrule
          R101-50              & 12                     & 50                     & 2,550                  & 50                   & 30,636                & 30,600               & 612 \cr
          R201-50              & 6                    & 50                     & 2,550                  & 50                                 & 15,318      & 15,300        & 306 \cr
           RC101-50             & 8                      & 50                     & 2,550                  & 50                   & 20,424               & 20,400               & 408 \cr
          RC201-50             & 5                      & 50                     & 2,550                  & 50                   & 12,765              & 12,750               & 255 \cr
\
      C101-100             & 10                     & 100                    & 10,100                 & 100                  & 101,030              & 101,000              & 1,010 \cr
       C201-100             & 3                      & 100                    & 10,100                 & 100                  & 30,309               & 30,300               & 303 \cr
      \bottomrule
    \end{tabular}}
\end{table}

\subsubsection{Performances of the Proposed Algorithm} In the subsequent tables,  the ``$f^*$" and ``$f$" columns correspond to the best-known objective values and the objective values of feasible solutions generated by these compared methods, respectively. The ``Time" column denotes the CPU time (in seconds) that the methods taken by the algorithms to meet the stopping criteria. The optimality gap is defined by gap$^1 =|f-f^*|/|f^*|$. Due to the lack of an inherent termination criterion in OR-tools, it runs for the entire limited time and outputs the corresponding feasible solution. Therefore, we do not report its solution time. Table \ref{tab:solomon1} shows the stability of the compared methods under various problem sizes on the C1-type instances. We can observe that our methods outperform the OR-Tools and Gurobi, and are more stable than the OR-Tools. From Table \ref{tab:solomon2}, we can find that the ALM-C algorithm outperforms the OR-Tools and Gurobi in most instances, achieving significantly lower optimality gaps and competitive computation times. The ALM-P algorithm also exhibits promising results, often outperforming Gurobi in efficiency and solution quality.
Overall, both the proposed ALM-C and ALM-P algorithms demonstrate their superiority and robustness in solving the CVRPTW problem, providing high-quality solutions with good computational efficiency when compared to existing solvers and heuristic approaches.

\begin{table*}[h]
    \small
    \caption{Performance comparison on Solomon's C1 instances with varying problem size and perturbation. A time limit of 500 seconds is set for OR-Tools.}
    \label{tab:solomon1}
        \centering
    \begin{tabular}{lll|r|ll|lll|lll|lll}
        \toprule
        \multirow{2}{*}{ins.}    & \multirow{2}{*}{$|V|$} & \multirow{2}{*}{$|F|$} & \multirow{2}{*}{$f^*$} & \multicolumn{2}{c}{OR-Tools} & \multicolumn{3}{c}{Gurobi} & \multicolumn{3}{c}{ALM-P} & \multicolumn{3}{c}{ALM-C}                                                                          \\
                                 &                        &                        &                        & $f$                          & gap$^1$              & $f$                        & gap$^1$              & Time    & $f$   & gap$^1$ & Time   & $f$   & gap$^1$ & Time   \\
        \midrule
        \multirow[t]{3}{*}{c101} & 25                     & 3                      & 191.3                  & 632.6                        & 230.7\%                    & 191.8                      & 0.3\%                      & 0.0    & 191.8 & 0.3\%         & 2.4   & 191.8 & 0.3\%        & 2.7   \\
                                 & 50                     & 5                      & 362.4                  & 514.4                        & 41.9\%                     & 363.2                      & 0.2\%                      & 1.1    & 363.2 & 0.2\%         & 5.9   & 363.2 & 0.2\%        & 5.0   \\
                                 & 100                    & 10                     & 827.3                  & 828.9                        & 0.2\%                      & 828.9                      & 0.2\%                      & 8.9    & 828.9 & 0.2\%         & 57.8  & 828.9 & 0.2\%        & 55.0  \\
        \midrule
        \multirow[t]{3}{*}{c102} & 25                     & 3                      & 190.3                  & 786.1                        & 313.1\%                    & 190.7                      & 0.2\%                      & 1.9    & 190.7 & 0.2\%         & 3.7   & 190.7 & 0.2\%        & 4.6   \\
                                 & 50                     & 5                      & 361.4                  & 667.9                        & 84.8\%                     & 362.2                      & 0.2\%                      & 25.5   & 362.2 & 0.2\%         & 16.0  & 362.2 & 0.2\%        & 15.3  \\
                                 & 100                    & 10                     & 827.3                  & 828.9                        & 0.2\%                      & -                          & \-                         & 2000 & 828.9 & 0.2\%         & 167.4 & 828.9 & 0.2\%        & 196.4 \\
        \midrule
        \multirow[t]{3}{*}{c103} & 25                     & 3                      & 190.3                  & 786.1                        & 313.1\%                    & 190.7                      & 0.2\%                      & 7.0    & 190.7 & 0.2\%         & 15.9  & 190.7 & 0.2\%        & 11.0  \\
                                 & 50                     & 5                      & 361.4                  & 667.9                        & 84.8\%                     & 362.2                      & 0.2\%                      & 1584.8 & 365.3 & 1.1\%         & 40.0  & 362.2 & 0.2\%        & 41.5  \\
                                 & 100                    & 10                     & 826.3                  & -                            & -                          & -                          & \-                         & 2000& 839.0 & 1.5\%         & 362.6 & 828.1 & 0.2\%        & 259.9 \\
        \midrule                 
        \multirow[t]{3}{*}{c104} & 25                     & 3                      & 186.9                  & 875.6                        & 368.5\%                    & 187.4                      & 0.3\%                      & 45.1   & 188.6 & 0.9\%         & 33.5  & 188.6 & 0.9\%        & 17.9  \\
                                 & 50                     & 5                      & 358.0                  & 606.2                        & 69.3\%                     & 360.1                      & 0.6\%                      & 2000 & 362.2 & 1.2\%         & 57.1  & 358.9 & 0.3\%        & 319.8 \\
                                 & 100                    & 10                     & 822.9                  & 1202.3                       & 46.1\%                     & -                          & \-                         & 2000 & 890.9 & 8.3\%         & 667.5 & 904.9 & 10.0\%        & 363.3 \\
        \midrule
        \multirow[t]{3}{*}{c105} & 25                     & 3                      & 191.3                  & 609.6                        & 218.6\%                    & 191.8                      & 0.3\%                      & 0.1    & 191.8 & 0.3\%         & 3.3   & 191.8 & 0.3\%        & 3.4   \\
                                 & 50                     & 5                      & 362.4                  & 482.4                        & 33.1\%                     & 363.2                      & 0.2\%                      & 0.3    & 363.2 & 0.2\%         & 6.4   & 363.2 & 0.2\%        & 6.0   \\
                                 & 100                    & 10                     & 827.3                  & 828.9                        & 0.2\%                      & 828.9                      & 0.2\%                      & 15.8   & 828.9 & 0.2\%         & 53.3  & 828.9 & 0.2\%        & 34.2  \\
        \midrule
        \multirow[t]{3}{*}{c106} & 25                     & 3                      & 191.3                  & 639.6                        & 234.3\%                    & 191.8                      & 0.3\%                      & 0.1    & 191.8 & 0.3\%         & 3.6   & 191.8 & 0.3\%        & 3.8   \\
                                 & 50                     & 5                      & 362.4                  & 430.4                        & 18.8\%                     & 363.2                      & 0.2\%                      & 1.4    & 363.2 & 0.2\%         & 5.8   & 363.2 & 0.2\%        & 5.9   \\
                                 & 100                    & 10                     & 827.3                  & 828.9                        & 0.2\%                      & 828.9                      & 0.2\%                      & 650.6  & 828.9 & 0.2\%         & 69.6  & 828.9 & 0.2\%        & 47.0  \\
        \midrule
        \multirow[t]{3}{*}{c107} & 25                     & 3                      & 191.3                  & 565.6                        & 195.6\%                    & 191.8                      & 0.3\%                      & 0.1    & 191.8 & 0.3\%         & 3.2   & 191.8 & 0.3\%        & 4.3   \\
                                 & 50                     & 5                      & 362.4                  & 457.4                        & 26.2\%                     & 363.2                      & 0.2\%                      & 0.8    & 363.2 & 0.2\%         & 6.1   & 363.2 & 0.2\%        & 6.6   \\
                                 & 100                    & 10                     & 827.3                  & 828.9                        & 0.2\%                      & 828.9                      & 0.2\%                      & 14.6   & 828.9 & 0.2\%         & 65.2  & 828.9 & 0.2\%        & 35.2  \\
        \midrule
        \multirow[t]{3}{*}{c108} & 25                     & 3                      & 191.3                  & 564.6                        & 195.1\%                    & 191.8                      & 0.3\%                      & 0.9    & 191.8 & 0.3\%         & 4.5   & 191.8 & 0.3\%        & 4.4   \\
                                 & 50                     & 5                      & 362.4                  & -                            & -                          & 363.2                      & 0.2\%                      & 18.6   & 363.2 & 0.2\%         & 8.4   & 363.2 & 0.2\%        & 14.2  \\
                                 & 100                    & 10                     & 827.3                  & -                            & -                          & -                          & \-                         & 2000 & 828.9 & 0.2\%         & 98.0  & 828.9 & 0.2\%        & 265.4 \\
        \midrule
        \multirow[t]{3}{*}{c109} & 25                     & 3                      & 191.3                  & 475.6                        & 148.6\%                    & 191.8                      & 0.3\%                      & 9.6    & 191.8 & 0.3\%         & 6.3   & 191.8 & 0.3\%        & 9.7   \\
                                 & 50                     & 5                      & 362.4                  & 363.2                        & 0.2\%                      & 363.2                      & 0.2\%                      & 1867.4 & 363.2 & 0.2\%         & 12.3  & 363.2 & 0.2\%        & 13.1  \\
                                 & 100                    & 10                     & 827.3                  & 828.9                        & 0.2\%                      & -                          & \-                         & 2000 & 828.9 & 0.2\%         & 234.5 & 828.9 & 0.2\%        & 304.3 \\
        \bottomrule
    \end{tabular}
\end{table*}
\begin{table*}[h]
    \small
    \caption{Complete results on other Solomon's instances, all instances are associated with 50 customers.  A time limit of 500 seconds is set for OR-Tools.}
    \label{tab:solomon2}
        \centering
    \begin{tabular}{ll|r|ll|lll|lll|lll}
        \toprule
        \multirow{2}{*}{ins.} & \multirow{2}{*}{$|J|$} & \multirow{2}{*}{$f^*$} & \multicolumn{2}{c}{OR-Tools} & \multicolumn{3}{c}{Gurobi} & \multicolumn{3}{c}{ALM-P} & \multicolumn{3}{c}{ALM-C}                                                                            \\
                              &                        &                        & $f$                          & $\varepsilon$              & $f$                        & $\varepsilon$              & $t$    & $f$    & $\varepsilon$ & $t$   & $f$    & $\varepsilon$ & $t$   \\
        \midrule
 c201.50 & 3 & 360.2 & 1622 & 350.3\% & 361.8 & 0.4\% & 0.1 & 361.8 & 0.4\% & 4.2 & 361.8 & 0.4\% & 3.4 \\ 
        c202.50 & 3 & 360.2 & 1622 & 350.3\% & 361.8 & 0.4\% & 5.7 & 361.8 & 0.4\% & 19.1 & 366.8 & 1.8\% & 7.1 \\ 
        c203.50 & 3 & 359.8 & 1713.4 & 376.2\% & 361.4 & 0.4\% & 113.5 & 361.4 & 0.4\% & 43.7 & 361.4 & 0.4\% & 76.8 \\ 
        c204.50 & 2 & 350.1 & 1435.3 & 310.0\% & 351.7 & 0.5\% & 2000 & 366.9 & 4.8\% & 55.7 & 351.7 & 0.5\% & 166.3 \\ 
        c205.50 & 3 & 359.8 & 1729 & 380.5\% & 361.4 & 0.4\% & 1.1 & 361.4 & 0.4\% & 7.9 & 361.4 & 0.4\% & 7.0 \\ 
        c206.50 & 3 & 359.8 & 1741.8 & 384.1\% & 361.4 & 0.4\% & 1.5 & 361.4 & 0.4\% & 14.4 & 361.4 & 0.4\% & 19.1 \\ 
        c207.50 & 3 & 359.6 & 1622.3 & 351.1\% & 361.2 & 0.4\% & 11.9 & 361.2 & 0.5\% & 25.0 & 361.2 & 0.4\% & 125.7 \\ 
        c208.50 & 2 & 350.5 & 1629.2 & 364.8\% & 352.1 & 0.5\% & 8.0 & 355.9 & 1.5\% & 14.2 & 352.1 & 0.5\% & 31.5 \\     \midrule 
        r101.50 & 12 & 1044 & 1541.9 & 47.7\% & 1046.7 & 0.3\% & 2.1 & 1046.7 & 0.3\% & 38.5 & 1046.7 & 0.3\% & 47.3 \\ 
        r102.50 & 11 & 909 & 1374.8 & 51.2\% & 911.4 & 0.3\% & 2000 & 939.5 & 3.4\% & 248.9 & 925.3 & 1.8\% & 317.2 \\ 
        r103.50 & 9 & 772.9 & 1069.1 & 38.3\% & - & - & 2000 & 806.4 & 4.3\% & 289.8 & 803.8 & 4.0\% & 256.1 \\ 
        r104.50 & 6 & 625.4 & 743.5 & 18.9\% & - & - & 2000 & 654.4 & 4.6\% & 472.1 & 686.6 & 9.8\% & 649.9 \\ 
        r105.50 & 9 & 899.3 & 1101.7 & 22.5\% & 901.9 & 0.3\% & 268.5 & 906.13 & 0.8\% & 61.1 & 901.9 & 0.3\% & 60.2 \\ 
        r106.50 & 5 & 793 & 937.8 & 18.3\% & - & - & 2000 & - & - & - & - & - & - \\ 
        r107.50 & 7 & 711.1 & 790.5 & 11.2\% & - & - & 2000 & 816.47 & 14.8\% & 120.5 & 741.5 & 4.3\% & 336.2 \\ 
        r108.50 & 6 & 617.7 & 698.1 & 13.0\% & - & - & 2000 & 645.7 & 4.5\% & 215.1 & 643.9 & 4.2\% & 466.5 \\ 
        r109.50 & 8 & 786.8 & 873.2 & 11.0\% & 805.6 & 2.4\% & 2000 & 796.3 & 1.2\% & 77.2 & 788.7 & 0.2\% & 324.9 \\ 
        r110.50 & 7 & 697 & 887.1 & 27.3\% & - & - & 2000 & 754.9 & 8.3\% & 560.6 & 768.1 & 10.2\% & 225.4 \\ 
        r111.50 & 7 & 707.2 & 784.5 & 10.9\% & - & - & 2000 & 750.0 & 6.0\% & 297.8 & 801.9 & 13.4\% & 220.8 \\ 
        r112.50 & 6 & 630.2 & 730.6 & 15.9\% & - & - & 2000 & 665.7 & 5.6\% & 476.5 & 698.0 & 10.8\% & 300.2 \\ \midrule  
        r201.50 & 6 & 791.9 & 1196.9 & 51.1\% & 794.3 & 0.3\% & 5.2 & 803.5 & 1.5\% & 27.0 & 801.3 & 1.2\% & 30.4 \\ 
        r202.50 & 5 & 698.5 & 1173.2 & 68.0\% & 723.0 & 3.5\% & 2000 & 735.7 & 5.3\% & 437.0 & 726.3 & 4.0\% & 337.2 \\ 
        r203.50 & 5 & 605.3 & 1173.2 & 93.8\% & 608.0 & 0.4\% & 2000& 639.4 & 5.6\% & 168.6 & 653.0 & 7.9\% & 380.9 \\ 
        r204.50 & 2 & 506.4 & 1127.9 & 122.7\% & 512.4 & 1.2\% & 2000 & 524.0 & 3.5\% & 75.8 & 583.8 & 15.3\% & 401.6 \\ 
        r205.50 & 4 & 690.1 & 1132.3 & 64.1\% & 700.2 & 1.5\% & 2000 & 732.4 & 6.1\% & 74.5 & 731.7 & 6.0\% & 87.8 \\ 
        r206.50 & 4 & 632.4 & 1066.9 & 68.7\% & 657.2 & 3.9\% & 2000 & 682.1 & 7.9\% & 188.4 & 677.2 & 7.1\% & 266.6 \\ 
        r207.50 & 3 & 361.6 & 1046.8 & 189.5\% & - & - & 2000 & 362.6 & 0.3\% & 13.1 & 364.1 & 0.7\% & 25.2 \\ 
        r208.50 & 1 & 328.2 & 1019.6 & 210.7\% & - & - & 2000 & 329.3 & 0.3\% & 11.9 & 329.3 & 0.3\% & 13.3 \\ 
        r209.50 & 4 & 600.6 & 1043.6 & 73.8\% & 639.6 & 6.5\% & 2000 & 608.5 & 1.3\% & 217.9 & 609.4 & 1.5\% & 58.5 \\ 
        r210.50 & 4 & 645.6 & 1119.5 & 73.4\% & 661.0 & 2.4\% & 2000 & 710.2 & 10.0\% & 288.6 & 669.2 & 3.7\% & 253.4 \\ 
        r211.50 & 3 & 535.5 & 958.7 & 79.0\% & - & - & 2000 & 590.4 & 10.3\% & 168.8 & 555.1 & 3.7\% & 339.8 \\ \midrule 
        rc101.50 & 8 & 944 & 1108.3 & 17.4\% & 945.6 & 0.2\% & 2000 & 945.6 & 0.2\% & 191.2 & 945.6 & 0.2\% & 78.8 \\ 
        rc102.50 & 7 & 822.5 & 940.7 & 14.4\% & - & - & 2000 & 823.1 & 0.1\% & 510.3 & 823.1 & 0.1\% & 242.9 \\ 
        rc103.50 & 6 & 710.9 & 834.9 & 17.4\% & - & - & 2000 & 736.4 & 3.6\% & 134.7 & 751.3 & 5.7\% & 335.2 \\ 
        rc104.50 & 5 & 545.8 & 641.4 & 17.5\% & - & - & 2000 & 546.5 & 0.1\% & 101.6 & 546.5 & 0.1\% & 91.5 \\
        rc105.50 & 8 & 855.3 & 1112.7 & 30.1\% & - & - & 2000 & 873.2 & 2.1\% & 99.0 & 902.7 & 5.5\% & 134.4 \\ 
        rc106.50 & 6 & 723.2 & 793 & 9.7\% & - & - & 2000 & 733.7 & 1.5\% & 79.3 & 728.1 & 0.7\% & 152.4 \\ 
        rc107.50 & 6 & 642.7 & 752.3 & 17.1\% & - & - & 2000 & 650.3 & 1.2\% & 235.2 & 644.0 & 0.2\% & 183.7 \\ 
        rc108.50 & 6 & 598.1 & 690.3 & 15.4\% & - & - & 2000 & 600.7 & 0.4\% & 256.1 & 599.2 & 0.2\% & 276.4 \\ \midrule 
        rc201.50 & 5 & 684.8 & 1904.7 & 178.1\% & 686.3 & 0.2\% & 12.5 & 687.7 & 0.4\% & 31.9 & 686.3 & 0.2\% & 14.8 \\ 
        rc202.50 & 5 & 613.6 & 1172.6 & 91.1\% & 615.0 & 0.2\% & 2000 & 615.6 & 0.3\% & 62.2 & 615.0 & 0.2\% & 85.9 \\ 
        rc203.50 & 4 & 555.3 & 1163.1 & 109.5\% & 556.5 & 0.2\% & 2000 & 590.4 & 6.3\% & 256.7 & 558.5 & 0.6\% & 265.8 \\ 
        rc204.50 & 3 & 444.2 & 1090.3 & 145.5\% & 509.4 & 14.7\% & 2000 & 451.8 & 1.7\% & 218.4 & 462.3 & 4.1\% & 153.7 \\ 
        rc205.50 & 5 & 630.2 & 1222.7 & 94.0\% & 632.0 & 0.3\% & 2000 & 632.0 & 0.3\% & 89.8 & 632.0 & 0.3\% & 56.0 \\ 
        rc206.50 & 5 & 610 & 1088.5 & 78.4\% & 611.7 & 0.3\% & 2000 & 611.7 & 0.3\% & 36.8 & 611.7 & 0.3\% & 42.9 \\ 
        rc207.50 & 4 & 558.6 & 998.3 & 78.7\% & - & - & 2000 & 591.7 & 5.9\% & 315.5 & 608.5 & 8.9\% & 281.1 \\ 
        rc208.50 & 2 & 269.1 & 936.9 & 248.2\% & - & - & 2000 & 269.6 & 0.2\% & 65.0 & 269.6 & 0.2\% & 109.1 \\ 
        \bottomrule
    \end{tabular}
\end{table*}

To demonstrate the advantages of our methods, we show the convergence curve of the primal bound and dual bound of Gurobi, comparing it to the solution found by our solver with the instance
``C109.50" as an example. As we can observe in Figure \ref{fig:comp_grb_alm}, our methods achieve near-optimal solutions in around 15 seconds, significantly faster than Gurobi which takes approximately 300 seconds. We notice that as the objective values increase, the corresponding constraint violation decreases simultaneously, facilitating fast convergence.

\begin{figure*}[ht]
  \centering
    \subfloat[Objective Value]{%
        \includegraphics[width=0.32\textwidth]{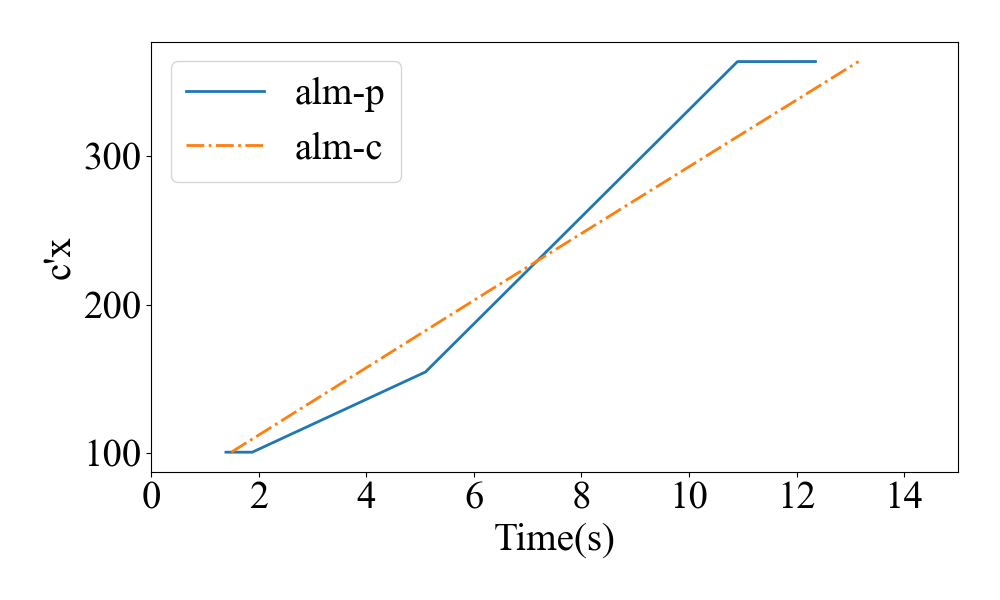}
        \label{fig:sub1}
    }
    \subfloat[Constraint Violation]{%
        \includegraphics[width=0.32\textwidth]{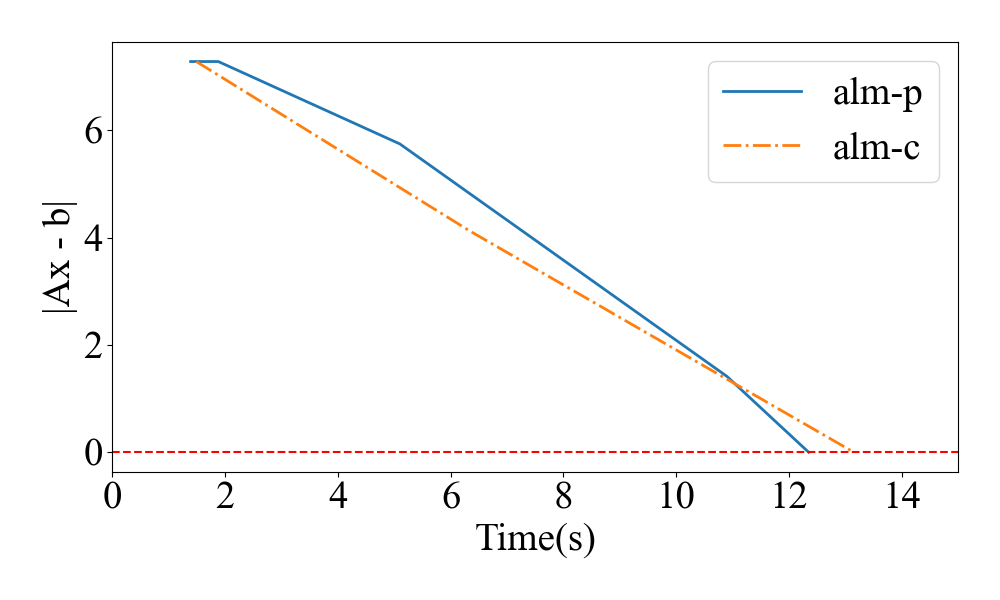}
        \label{fig:sub2}
    }
    \subfloat[Primal and Dual Bound of Gurobi]{%
        \includegraphics[width=0.32\textwidth]{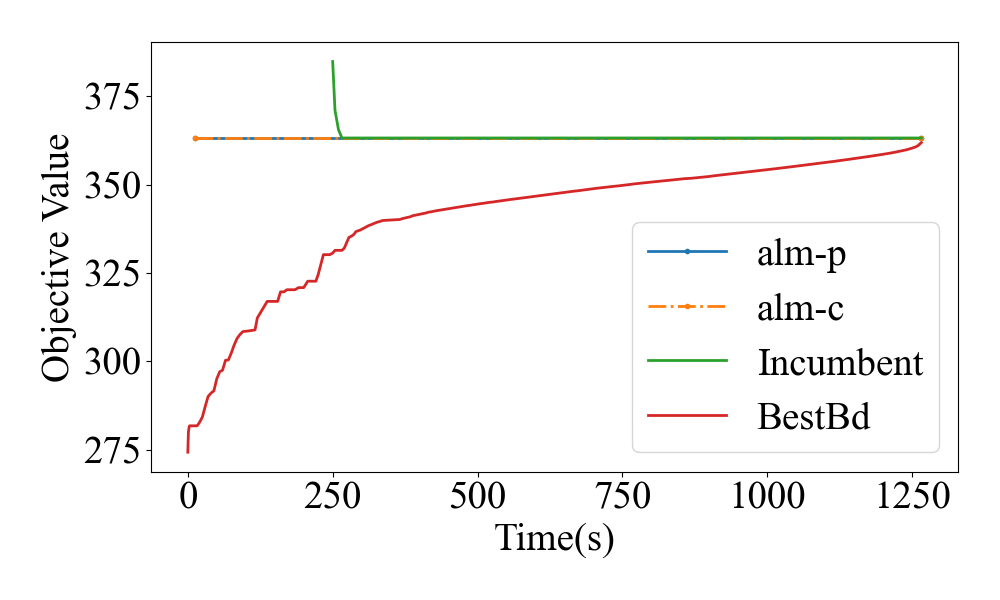}
        \label{fig:sub3}
    }
  \caption{Comparison of Gurobi and our methods}
    \label{fig:comp_grb_alm}
\end{figure*}

\subsection{Train timetabling problem}
We consider following space-time network model for the train timetabling problem (TTP) on a macro level, which is based on the model in \cite{caprara2002modeling}. We use a directed, acyclic and multiplicative graph \(G = (V, E)\) to characterize the train timetabling problem, where
$V$ and $E$  denote the set of all nodes and the set of all arcs. For each train \(j\in \N_p\), the sets or parameters with superscript or subscript notation corresponds to relevant object to $j$.  For each arc $e \in E_j$, we introduce a binary variable $x_e$ equal to 1 if the arc $e$ is selected. For each node $v \in V$, let $\delta^+_j(v)$  and  $\delta^-_j(v)$ be the sets of arcs in $E_j$ leaving and entering node $v$, respectively.
Then the integer programming model of TTP is given by
\begin{subequations}\label{STMN}
  \begin{align}
    \label{eq.g.obj}    \max           & \sum\nolimits_{j \in \N_p} \sum\nolimits_{e \in E^j} p_e x_e                                                                                \\
    \label{eq.g.start}     \text{s.t.} & \sum\nolimits_{e \in \delta_{j}^{+}(\sigma)} x_e \le 1, \hspace{3.1cm} j \in \N_p                                                          \\
    \label{eq.g.flow}                  &  \sum_{e \in \delta_{j}^{-}(v)} x_e=\sum_{e \in \delta_{j}^{+}(v)} x_e, ~ j \in \N_p, \hspace{2.2mm}  v \in V \backslash\{\sigma, \tau\} \\
    \label{eq.g.end}                   &  \sum\nolimits_{e \in \delta_{j}^{-}(\tau)} x_e \le 1, \hspace{3.2cm} j \in \N_p                                                             \\
    \label{eq.g.interval}              &  \sum\nolimits_{v' \in \mathcal N(v)} \sum\nolimits_{j \in \mathcal{T}(v')} \sum\nolimits_{e \in \delta_{j}^{-}(v')} x_e \le 1,  v \in V         \\
    \label{eq.g.clique}                &  \sum\nolimits_{e \in C} x_e \le 1, \hspace{3.7cm} C \in \mathcal{C}                                                                         \\
    \label{eq.g.binary}                &  x_e \in \{0, 1\}, \hspace{4.25cm} e \in E,
  \end{align}
\end{subequations}
where $p_e$ is the ``profit'' of using a certain arc $e$. $\sigma$ and $\tau$ denote artificial origin and destination nodes, respectively.   $\mathcal{T}(v)$ and $\mathcal{N}(v)$ denote the set of trains may passing through node $v$ and the set of nodes conflicted with node $v$, respectively. $\mathcal{C}$ denotes the (exponentially large) family of maximal subsets $C$ of pairwise incompatible arcs. In this model, \eqref{eq.g.start}, \eqref{eq.g.flow}, \eqref{eq.g.end} imply the arcs of train $j$ should form a valid path in \(G\), \eqref{eq.g.interval} represents headway constraints, \eqref{eq.g.clique} forbids the simultaneous selection of incompatible arcs, imposing the track capacity constraints.

Let $\x_j = \{x_e \mid e \in  E_j\}$. Then we can rewrite this space-time network model in the general form as (\ref{bLP}).  
Our goal is to show that the proposed Algorithm \ref{PBALM} is fully capable to provide implementable time tables for the Jinghu railway. Specifically, our algorithms are tested on the time-tabling problem for Beijing-Shanghai high-speed railway (or Jinghu high-speed railway in Mandarin). As one of the busiest railways in the world, the Beijing-Shanghai high-speed railway transported over 210 million passengers in 2019\footnote{For details, see https://en.wikipedia.org/wiki/Beijing-Shanghai\_high-speed\_railway.}. In our test case, the problem consists of 29 stations and 292 trains in both directions (up and down), including two major levels of speed: 300 km/h and 350 km/h. Several numerical experiments are carried out based on the data of Beijing-Shanghai high-speed railway to demonstrate the feasibility and effectiveness of the proposed strategy.

We compare our proposed methods with the Gurobi solver and ADMM \cite{ZHANG2020102823} on small and real-world instances, as presented in Tables \ref{small_network} and \ref{mpnsize}. The notations $|F|, |S|$ and $|T|$ represent the number of trains, the number of stations, and the time window.  In most large-scale cases, the Gurobi solver takes a much longer time to solve, so we set a time limit of two hours.   Since \eqref{eq.g.obj} maximizes the positive revenue, we have a negative cost if reversing the objective to a minimization problem, then both subproblems \eqref{xjko} and \eqref{xjkl} are solved by the Bellman-Ford algorithm, which is an efficient tool for solving the shortest path problem.

\subsubsection{Performances on small instances}
We first validate our algorithm on a smaller sub-network using a subset of stations of the Jinghu railway.  We set the revenue of each arc proportional to the distance between two stations, which corresponds to an intuition that longer rides should bear higher incomes.
For our proposed methods, we set uniformly $\rho^0=20$ and $\sigma=1.2$ for all instances. We set a time limit of 100 seconds for our algorithms. Since the optimal value is unknown, we report the upper bounds (UB) obtained by the Gurobi solver as a reference for comparison in Table \ref{small_network}. We can observe that our ALM-C performs competitively with the Gurobi solver in terms of both the optimal value and computation time.

\begin{table*}[h]
  \caption{Performance on four small examples on sub-networks of the Jinghu railway for maximizing the total revenue}
  \centering
  \label{small_network}
  \begin{tabular}{cccccccccccc}
    \toprule
   \multirow{2}{*}{ $|F|$ } & \multirow{2}{*}{$|S|$} & \multirow{2}{*}{$|T|$} & \multicolumn{3}{c}{Gurobi} & \multicolumn{2}{c}{ADMM} & \multicolumn{2}{c}{ALM-P} & \multicolumn{2}{c}{ALM-C}                                                                      \\
    \cmidrule(lr){4-6} \cmidrule(lr){7-8} \cmidrule(lr){9-10} \cmidrule(lr){11-12}
                           &                        &                        &UB                         & $f$                       & Time                       & $f$                         & Time  & $f$              & Time  & $f^*$              & Time        \\
    \midrule
    20                       & 15                     & 200                    & 20952            & \textbf{20952}          & \textbf{5.0}              & \textbf{20952}            & 14.2 & \textbf{20952} & 13.1 & \textbf{20952} & 10.0          \\
    25                       & 15                     & 200                    & 23396           & \textbf{23396}                   & 11.5                      & 23396                     & 31.3 & 21085          & 22.1 & \textbf{23396} & \textbf{11.4} \\
    30                       & 15                     & 300                               & 25707    & \textbf{25707}                & 14.4                      & 24552                     & 14.1 & \textbf{25707} & 23.0 & \textbf{25707} & \textbf{9.2}  \\
    40                       & 15                     & 300                               & 34305      & \textbf{34305}             & 24.3                      & 31817                     & 100.0   & 32061          & 100.0   & \textbf{34305} & \textbf{17.0} \\
    \bottomrule
  \end{tabular}
\end{table*}

\subsubsection{Performances on real-world instances}

To verify the efficiency of the ALM-C and ALM-P on large-scale data with all stations involved, we consider the following five examples of ascending problem size in Table \ref{mpnsize}.  Similarly, we try to maximize the total revenue of our schedule. In practice, the revenue of an arc may be ambiguous. Besides, since Jinghu high-speed railway always has a high level of utilization, introducing new trains is always beneficial since it further covers unmet demand. In this view, we introduce an alternative objective function to simply maximize the number of scheduled trains.  Perhaps not surprisingly, this simplified objective can dramatically speed up our methods for practical interest. This is due to the fact that the coefficient $p_e$ in the objective function is sparse, resulting in faster identification of the optimal solution under this model.

 For our proposed methods, we set $\rho^0=10^{-3}$ and $\sigma=2$ for the instances No. 1-2, $\rho^0=10^{-2}$ and $\sigma=1.1$ for the instances No. 3 and No. 5, $\rho^0=10^{-3}$ and $\sigma=1.1$ for the instance No. 4. We set $\x^0$ and $\lambda^0$ to be zero for all instances.
Both the number of variables and the number of constraints of the model \eqref{STMN} are very large for this practical TTP instance, whose dimensions are up to tens of millions.
Tables \ref{tab:performance_BCD_LU_rev} and \ref{tab:performance_BCD_LU} show the performance results of our proposed methods and the Gurobi solver, where gap$^2:= (\UB-f)/f$.

\begin{table}[h]
  \caption{Five examples of the large practical railway network}
  \label{mpnsize}
  \centering
  \begin{tabular}{ccccccc}
    \toprule
    \multirow{1}{*}{No.} & \multirow{1}{*}{$|F|$} & \multirow{1}{*}{$|S|$} & \multirow{1}{*}{$|T|$} & \multirow{1}{*}{$m$} & \multirow{1}{*}{$q$} & \multirow{1}{*}{$n$}\cr
    \midrule
    1                    & 50                     & 29                     & 300                    & 49,837               & 231,722              & 308,177 \cr
    2                    & 50                     & 29                     & 600                    & 113,737              & 1,418,182            & 1,396,572 \cr
    3                    & 100                    & 29                     & 720                    & 145,135              & 1,788,088            & 4,393,230 \cr
    4                    & 150                    & 29                     & 960                    & 199,211              & 3,655,151            & 9,753,590 \cr
    5                    & 292                    & 29                     & 1080                   & 228,584              & 13,625,558           & 26,891,567 \cr
    \bottomrule
  \end{tabular}
\end{table}

\renewcommand{\arraystretch}{1} 
\begin{table*}[h]
  \caption{Performance comparison between the Gurobi solver, ADMM, ALM-P and ALM-C on the large networks for maximizing total revenue in Table \ref{mpnsize}.
  }
  \label{tab:performance_BCD_LU_rev}
    \centering
  \begin{threeparttable}
    \centering
    \begin{tabular}{ccccccccccccccc}
      \toprule
      \multirow{2}{*}{No.} & \multicolumn{3}{c}{ Gurobi} & \multicolumn{3}{c}{ADMM} & \multicolumn{3}{c}{ ALM-P} & \multicolumn{3}{c}{ ALM-C}\cr
      \cmidrule(lr){2-4} \cmidrule(lr){5-7} \cmidrule(lr){8-10} \cmidrule(lr){11-13}
                          & $f$                       & gap$^2$                         & Time                           & $f$        & gap$^2$     & Time         & $f$        & gap$^2$     & Time & $f$                 & gap$^2$              & Time \cr
      \midrule
      1                   & 6.18e+04                & 13.6\%                     & 7200.0                             & 6.94e+04 & 3.0\%  & \textbf{54.1} & 6.93e+04 & 3.1\%  & 84.3  & \textbf{6.98e+04} & \textbf{2.5\% } & 72.2\cr
      2                   & 7.61e+04                & 45.3\%                     & 7200.0                             & 1.34e+05 & 3.5\%  & \textbf{96.2} & 1.34e+05 & 3.4\%  & 144.4  & \textbf{1.35e+05} & \textbf{3.0\% } & 126.1    \cr
      3                   & 9.69e+04                & 61.5\%                     & 7200.0                             & 2.33e+05 & 7.3\%  & 558.3          & 2.36e+05 & 6.3\%  & 414.9  & \textbf{2.36e+05} & \textbf{6.3\% } & \textbf{396.6} \cr
      4                   & 1.23e+05                & 74.5\%                     & 7200.0                             & 4.38e+05 & 9.6\%  & 1422.7         & 4.42e+05 & 8.7\%  & 924.3 & \textbf{4.45e+05} & \textbf{8.1\% } & \textbf{714.9} \cr
      5                   & 2.95e+05                & 74.0\%                     & 7200.0                            & 9.15e+05 & 19.5\% & 3600.0         & 9.60e+05 & 15.6\% & 2094.2 & \textbf{9.63e+05} & \textbf{15.3\%} & \textbf{1722.7} \cr
      \bottomrule
    \end{tabular}
  \end{threeparttable}
\end{table*}

\renewcommand{\arraystretch}{1} 
\begin{table*}[h]
  \caption{Performance comparison between the Gurobi solver, ADMM, ALM-P and ALM-C on the large networks for maximizing the number of trains in the timetable in Table \ref{mpnsize}}
  \label{tab:performance_BCD_LU}
    \centering
  \begin{threeparttable}
    \centering
    \begin{tabular}{ccccccccccccccc}
      \toprule
      \multirow{2}{*}{No.} & \multirow{2}{*}{$f^*$} & \multicolumn{3}{c}{ Gurobi} & \multicolumn{3}{c}{ADMM} & \multicolumn{3}{c}{ ALM-P} & \multicolumn{3}{c}{ ALM-C}\cr
      \cmidrule(lr){3-5} \cmidrule(lr){6-8} \cmidrule(lr){9-11} \cmidrule(lr){12-14}
                           &                     & $f$                          & gap$^1$                      & Time                        & $f$                             & gap$^1$  & Time    & $f$            & gap$^1$ & Time          & $f$            & gap$^1$ & Time \cr
      \midrule
      1                    & --                  & \textbf{27}                & --                       & 3836.0                     & \textbf{27}                   & --   & 3.4    & \textbf{27}  & --  & \textbf{3.4} & \textbf{27}  & --  & 3.5 \cr
      2                    & 50                 & \textbf{50}                & \textbf{0}                        & 2965.8                     & \textbf{50}                   & \textbf{0}    & 8.2    & \textbf{50}  & \textbf{0}   & \textbf{8.1} & \textbf{50}  & \textbf{0}   & 8.2   \cr
      3                    & 100                & 11                         & 89.0\%                    & 7200.0                     & 89                            & 11.0\%   & 3001.4 & \textbf{100} & \textbf{0}   & 127.4        & \textbf{100} & \textbf{0}   & \textbf{96.4} \cr
      4                    & 150                & 14                         & 90.7\%                    & 7200.0                     & 148                           & 1.3\% & 3005.3 & \textbf{150} & \textbf{0}   & 99.5         & \textbf{150} & \textbf{0}   & \textbf{96.4} \cr
      5                    & 292                & 30                         & 89.7\%                    & 7200.0                     & 286                           & 2.1\% & 3003.8 & \textbf{292} & \textbf{0}   & 203.1        & \textbf{292} & \textbf{0}   & \textbf{172.1} \cr
      \bottomrule
    \end{tabular}
  \end{threeparttable}
\end{table*}

The results clearly demonstrate the high effectiveness of our methods compared to Gurobi, especially when dealing with large-scale data.  For the instance No. 1, both ALM-C and ALM-P can schedule 27 trains in a few seconds, which are at least 1000 times faster than the Gurobi solver and meanwhile obtain satisfactory accuracy performance. In particular, when the scale of data is up to tens of millions in instance No. 5, our ALM-C and ALM-P successfully schedule all 292 trains. In contrast, the Gurobi solver produces a relatively small number of trains, only 30 trains, and takes much longer. The results of other instances also illustrate the effectiveness of this technique. Therefore, we can conclude that the customized ALM provides a fast and global optimal solution for this practical problem.

\section{Conclusion}

In this paper, we study general integer programming with block structure. 
Benefiting from its special structure,  we extend the augmented Lagrangian method, originally designed for continuous problems, to effectively solve the problem \eqref{bLP}. By introducing a novel augmented Lagrangian function, we establish the strong duality and optimality for the problem \eqref{bLP}. Furthermore,  we provide the convergence results of the proposed methods for both the augmented Lagrangian relaxation and dual problems. To obtain high-quality feasible solutions, we develop a customized ALM combined with refinement techniques to iteratively improve the primal and dual solution quality simultaneously. The numerical experiments demonstrate that the customized ALM is time-saving and performs well for finding optimal solutions to a wide variety of practical problems.


%



\ifCLASSOPTIONcaptionsoff
  \newpage
\fi



%
\bibliographystyle{IEEEtran}
\bibliography{ref}

\end{document}